\documentclass[12pt]{amsart}
\usepackage{geometry}                
\usepackage{graphicx}

\usepackage{amsmath}
\usepackage{amsfonts}
\usepackage{enumerate}
\usepackage{amssymb}
\usepackage{epstopdf}

\newtheorem{theorem}{Theorem}[section]
\newtheorem{proposition}[theorem]{Proposition}
\newtheorem{lemma}[theorem]{Lemma}

\theoremstyle{definition}
\newtheorem{definition}[theorem]{Definition}

\numberwithin{equation}{section}
\newcommand{\C}{686}

\newcommand{\D}{664}

\begin{document}

\title{Contracting the boundary of a Riemannian 2-disc}

\author[Y.~Liokumovich]{Yevgeny Liokumovich}
\address {Department of Mathematics, University of Toronto, Toronto, Canada}
\email{e.liokumovich@utoronto.ca}

\author[A.~Nabutovsky]{Alexander Nabutovsky}
\address {Department of Mathematics, University of Toronto, Toronto, Canada}
\email{alex@math.toronto.edu}

\author[R.~Rotman]{Regina Rotman}
\address {Department of Mathematics, University of Toronto, Toronto, Canada}
\email{rina@math.toronto.edu}

\begin{abstract}
Let $D$ be a Riemannian 2-disc of area $A$, diameter $d$ and length of the boundary $L$.
We prove that it is possible to contract the boundary of $D$
through curves of length $\leq L + 200d\max\{1,\ln {\sqrt{A}\over d} \}$. This answers a twenty-year old question
of S. Frankel and M. Katz, a version of which was asked earlier by M.Gromov.

We also prove that a Riemannian $2$-sphere $M$ of diameter $d$ and area $A$
can be swept out by loops based at any prescribed point $p\in M$ of length
$\leq 200 d\max\{1,\ln{\sqrt{A}\over d} \}$. This estimate is optimal
up to a constant factor.
In addition, we provide much better (and nearly optimal) estimates for these
problems in the case, when $A<<d^2$. Finally, we describe the applications of
our estimates for study of lengths of various geodesics between a fixed
pair of points on ``thin" Riemannian $2$-spheres.
\end{abstract}

\date{}

\maketitle

\section{Main results}

Consider a 2-dimensional disc $D$ with a Riemannian metric. M. Gromov asked if there exists a
universal constant $C$, such that the boundary of $D$ could be homotoped to a point
through curves of length less than $C\max\{|\partial D|,  diam(D)\}$, where $|\partial D|$ denotes the length of
the boundary of $D$ and $diam (D)$ denotes its diameter. This question is a Riemannian analog of the well-known (and still
open) problem in
geometric group theory asking about the relationship between the filling length and the filling diameter (see \cite{Gr93}).

S. Frankel and M. Katz answered the question posed
by Gromov negatively in \cite{FK}.
They demonstrated that there is no
upper bound for lengths of curves in an ``optimal" homotopy contracting $\partial D$ in terms of
$\vert \partial D\vert$ and $diam(D)$.
Then they asked if there exists such an upper bound if one is allowed to use the area $Area(D)$ of $D$ in addition
to $\vert\partial D\vert$ and $diam (D)$.
In this paper we will prove
that the answer to this question is positive, and, moreover, provide nearly optimal upper bounds for lengths
of curves in an ``optimal" contracting homotopy in terms of $\vert\partial D\vert, diam(D)$ and $Area(D)$.
Note that S. Gersten and T. Riley ([GerR])
proved a similarly looking result in the context of geometric group theory. Yet in the Riemannian setting their approach seems to yield
an upper bound with the leading terms $const(\vert\partial D\vert+ diam(D)\max\{1,\ln{\sqrt{Area(D)}\over inj(D)}\})$, where $inj(D)$ denotes the injectivity radius of the disc, and so does not
lead to a solution of the problem posed by Frankel and Katz.

Define the {\it homotopy excess},  $exc(D)$, of a Riemannian disc $D$ as the infimum
of all values of 
$x$ such that for every $p\in\partial D$ the boundary of $D$ is contractible
to $p$ via loops of length $\leq |\partial D|+x$ based at $p$.
Let $exc(d,A)$ denote the supremum of $exc(D)$ over all discs $D$ of area $\leq A$ and diameter $\leq d$.
The examples of [FK] imply the existence of a positive constant $const$ such
that $exc(d,A)\geq const\ d\max\{1,\ln{\sqrt{A}\over d}\}$. The first of our main results
implies that this lower bound is optimal up to a constant factor:
\medskip

\noindent
{\bf Main Theorem A.} $$exc(d,A)\leq 200 d\max\{1,\ \ln{\sqrt{A}\over d}\}$$
\medskip

In fact, we are able to prove that
$\lim\sup_{{\sqrt{A}\over d}\longrightarrow
\infty} {exc(d,A)\over d\ln {\sqrt{A}\over d}}\leq {12\over \ln {3\over 2}}<30$
(see the remark after the proof of Theorem 1.2 in section 7).
On the other hand, using a modification of the examples of Frankel and Katz we were able to
prove that:

\medskip
\noindent
\begin{theorem}\label {AA} {\bf (Theorem AA.)}
\par\noindent
1. $$exc(d,A)\geq {\sqrt{3}\over 2\ln 3} d(\ln{\sqrt{A}\over d}- o(1))>0.788 d\ln {\sqrt{A}\over d},$$
as ${\sqrt{A}\over d}\longrightarrow\infty$.
\par\noindent
2. On the other hand, if ${\sqrt{A}\over d}$ is sufficiently small, then $$ exc(d,A)> 2d + 0.0825\sqrt{A}.$$
\end{theorem}

Our proof of Theorem AA relies on the recent work of G. Chambers and the third author [CR].


Here are some other upper estimates:

\begin{theorem} \label{main}
For any Riemannian 2-disc $D$ and a point
$p \in \partial D$ there exists a homotopy $\gamma_t$ of loops based at $p$ with $\gamma_0 = \partial D$
and $\gamma_1=\{ p \}$,
such that
$$|\gamma_t| \leq 2 |\partial D| + \C \sqrt{Area(D)} + 2 diam(D)$$
for all $t \in [0,1]$.
\end{theorem}

It easy to see that any upper bound for the lengths of $\vert\gamma_t\vert$ should be greater than $2diam(D)$. Therefore the upper bound provided by Theorem 1.1
is optimal for fixed values of $Area(D)$ and $\vert\partial D\vert$, when $diam(D)\longrightarrow\infty$. However, the next theorem provides a better bound,
when $Area(D)\longrightarrow\infty$ or $\vert\partial D\vert\longrightarrow\infty$ and immediately implies Main Theorem A stated above.

\begin{theorem} \label{main2}
For any Riemannian 2-disc $D$ and a point
$p \in \partial D$ there exists a homotopy $\gamma_t$ of loops based at $p$ with $\gamma_0 = \partial D$
and $\gamma_1=\{ p \}$,
such that
$$|\gamma_t| \leq |\partial D| + 159 diam(D)+40diam(D)\max\{0,\ln{\sqrt{Area(D)}\over diam(D)}\}$$
for all $t \in [0,1]$.
\end{theorem}


As a consequence of the previous theorems we obtain related results about diastoles of Riemannian
2-spheres $M$. A $diastole$ of $M$ was defined by F. Balacheff and S. Sabourau
in \cite{BS} as
$$dias(M) = \inf_{(\gamma_t)} \sup_{0 \leq t \leq 1} |\gamma_t|$$
where $(\gamma_t)$ runs over continuous families of {\it free loops} sweeping-out $M$. More precisely, the family  $(\gamma_t)$ corresponds to
a generator of $\pi_1 (\Lambda M, \Lambda^0 M)$, where $\Lambda M$ denotes the space
of free loops on $M$ and $\Lambda ^0 M$ denotes the space of constant loops.

In [S, Remark 4.10] S.Sabourau gave an example of Reimannain two-spheres with
arbitrarily large ratio $\frac{dias(M_n)}{\sqrt{Area(M_n)}}$.
In \cite{L} the first author gave an example of Riemannian two-spheres $M_n$ with
arbitrarily large ratio $\frac{dias(M_n)}{diam(M_n)}$. We show that if
both the diameter and the area of $M$ are bounded,
the  diastole can not approach infinity. Moreover, for every $p\in M$
one can define $Bdias_p(M)$ by the formula
$$Bdias_p(M) = \inf_{(\gamma_t)} \sup_{0 \leq t \leq 1} |\gamma_t|,$$
where $(\gamma_t)$ runs over continuous families of {\it loops based at $p$} sweeping-out $M$. Now define the {\it base-point diastole} $Bdias(M)$ as
$\sup_{p\in M} Bdias_p(M)$. It is clear that
$Bdias(M)\geq dias(M)$.
We prove the following
inequalities:

\begin{theorem} \label{sphere} {\bf (Main Theorem B.)}
For any Riemannian 2-sphere $M$ we have
$$ A.\ \ \  Bdias(M)\leq 664\sqrt{Area(M)} + 2 diam(M);$$
$$B.\ \ \  Bdias(M)\leq 159 diam(M)+40diam(M)\max\{0,\ln{\sqrt{Area(M)}\over diam(M)}\}.$$
Moreover,
as ${diam(M)\over\sqrt{Area(M)}}\longrightarrow 0$,
$$\ \ \  Bdias(M)\leq ({12\over\ln{3\over 2}}+o(1))diam(M)\ln{\sqrt{Area(M)}\over diam(M)}.$$
Furthermore, for each $p\in M$ there exists a sweep-out of $M$ by simple loops based at $p$ that pairwise
intersect only of $p$ and have lengths satsfying the upper bounds in the right hand sides of the above inequalities.
\end{theorem}

{\bf Remark.} This theorem had been used in [LNR2] to prove a better bound for $dias(M)$, when $\sqrt{Area(M)}<<diam(M)$, where we proved that
if $diam(M)>\sqrt{3}\sqrt{Area(M)}$, then $dias(M)\leq diam(M)+700\sqrt{Area(M)}.$
This fact was established as a part of the proof of Theorem 1.3 in [LNR2]. 
(However, it had not been explicitly mentioned there.)

\medskip

We also noticed that one can modify the examples of
[FK], [L] and [L2]  to construct sequences of Riemannian
manifolds $M_i$ diffeomorphic to $S^2$  demonstrating that our upper bounds for $Bdias(M)$ and even for $dias(M)$ given by Theorem B are optimal up to a constant factor.
More precisely, we prove the following theorem:

\begin{theorem} \label{BB} {\bf (Theorem BB.)}
1.  There exists a sequence of Riemannian metrics on
$S^2$ such that ${diam(M_i)\over\sqrt{Area(M_i)}}\longrightarrow 0$ and 
$dias(M_i) \geq {\sqrt{3}\over 2\ln 3}diam(M_i)\ln{\sqrt{Area(M_i)}\over diam(M_i)} (1-o(1)) > 0.788 diam(M_i)\ln {\sqrt{Area(M_i)}\over diam(M_i)}.$
\par\noindent
2. There exists a sequence of Riemannian metrics on
$S^2$ such that ${diam(M_i)\over\sqrt{Area(M_i)}}\longrightarrow\infty$
and $dias(M_i) > diam(M_i) + 0.0476 \sqrt{Area(M_i)}$ for 
all values of $i$. 
\end{theorem}
\par
The proof of Theorem BB is similar to the proof of Theorem AA, but one needs to use a more difficult theorem by G. Chambers and
the first author from [CL] instead of the monotonicity theorem from [CR] used in the proof of Theorem AA.  One can combine the second part of this theorem
with the remark after the text of Theorem B and conclude that, when ${\sqrt{Area(M)}\over diam(M)}\longrightarrow 0$, $dias(M)=diam(M)+O(\sqrt{Area(M)})$,
and the dependance on $Area(M)$ in the second term cannot be improved.

Note that in \cite{BS} F.Balacheff and S.Sabourau show that if 1-parameter families of loops in the definition of
the diastole are replaced with 1-parameter families of one-cycles, then for every Riemannian surface
$\Sigma$ of genus $g$ the resulting homological diastole
$dias_Z(\Sigma)$ satisfies

$$dias_Z (\Sigma) \leq 10^8(g+1) \sqrt{Area(\Sigma)}.$$

%
The proof of Theorem \ref{main} will proceed first by considering subdiscs of $D$ of small area
and then obtaining the general result for larger
and larger subdiscs by induction. The parameter of the induction will
be $\lfloor\log_{4\over 3}{Area\ D'\over \epsilon(D)}\rfloor$, where
$D'$ denotes a (variable) subdisc and $\epsilon(D)>0$ is very small.
(In particular, $\epsilon(D)$ is much smaller than the injectivity radius of $D$.)
As it is the case with many inductive arguments, it is more convenient
to prove a stronger statement. To state this stronger version of Theorem \ref{main}
we will need the following notation:

\begin{definition}
For each $p\in D$ $d_p(D)= max \{dist (p,x) | x \in D \}$.
Let $d_D= max \{d_p(D)| p \in \partial D \}$.
\end{definition}

From the definition we see that $d_D \leq diam (D)$.

If $l_1$ and $l_2$ are two non-intersecting simple paths between points $p$ and $q$ of $D$, then
$l_1 \cup -l_2$ is a simple closed curve bounding a disc $D' \subset D$.
We will show that there exists a path homotopy from $l_1$ to $l_2$ such that
the lengths of the paths in this homotopy are bounded
in terms of the area, the diameter and the length of the boundary of $D'$.

\begin{definition}
Let $D$ be a Riemannian disc and $D' \subset D$ be a subdisc.
Define a \textit{relative path diastole} of $D'$ as
$$pdias(D',D)= \text{sup}_{p,q \in \partial D'} \text{inf}_{(\gamma_t)} \text{sup}_{t \in [0,1]} |\gamma_t| $$
where $(\gamma_t)$ runs over all families of paths from $p$ to $q$ $\gamma_t: [0,1] \rightarrow D$
with $\gamma_t(0)=p$, $\gamma_t(1)=q$, where $\gamma_0=l_1$ and $\gamma_1= l_2$ are subarcs of $\partial D'=l_1 \cup -l_2$ intersecting only at their
endpoints $p$, $q$.
Let $pdias(D)=pdias(D,D)$.
\end{definition}

\begin{theorem} \label{main'}
A. For any Riemannian 2-disc $D$ with $|\partial D| \leq 2 \sqrt{3} \sqrt{Area(D)}$
$$ pdias(D) \leq |\partial D| + \D \sqrt{Area(D)} + 2d_{D}.$$

B. For any Riemannian 2-disc $D$ with $|\partial D| \leq 6 \sqrt{Area(D)}$
$$ pdias(D) \leq |\partial D| + \C \sqrt{Area(D)} + 2d_{D}.$$

C. For any Riemannian 2-disc $D$ with $|\partial D| > 6 \sqrt{Area(D)}$

$$ pdias(D) \leq |\partial D|+
2\lceil \log_{\frac{4}{3}}(\frac{|\partial D|-4\sqrt{Area(D)}}{2\sqrt{Area(D)}})  \rceil \sqrt{Area(D)}
+ \C \sqrt{Area(D)} + 2d_{D}$$$$\leq 2|\partial D| +\C \sqrt{Area(D)}+2d_D.$$
D. Also, if $d\geq 3\sqrt{A}$,
$$exc(d,A)\leq 3d+{2\over\ln{4\over 3}}\sqrt{A}\ln({2\over 3}({2d\over\sqrt{A}}-4))+\C \sqrt{A}.$$
\end{theorem}

Of course, Theorem \ref{main} immediately follows from Theorem \ref{main'} C.
The second inequality in Part C of the theorem can be easily proven
 by observing that 
$\frac{2 \ln(\frac{2(x-4)}{3})}{(\ln\frac{4}{3})x}<0.9735<1$
for $x \in [6,\infty]$. Setting $x=\frac{|\partial D|}{\sqrt{Area(D)}}$
we obtain the desired inequality. The last inequality provides a much better
upper bound for $exc(d,A)$, when $\sqrt{A}<<d$ and implies
that $\lim_{{\sqrt{A}\over d}\longrightarrow 0}exc(d,A)\leq 3d$.

\noindent

As $exc(d,A)\geq 2d$, it is natural to ask the following question:

{\bf Question.} Is it true that, when $d$ is fixed, and $A\longrightarrow 0$, $exc(d,A)=2d+O(\sqrt{A})?$

In the first version of this paper ([LNR], v.1) we posed this question as an open problem. It was later resolved by P. Papasoglu in [P2].
He presented a beautiful and ingenuous argument demonstrating that our Theorem 1.1 imples that $exc(d,A)\leq 2d+1000\sqrt{A}$.
Comparing his result with the second part of Theorem AA we see that $exc(d,A)=2d +O(\sqrt{A})$, as $A=o(d^2)$, and the dependence on $A$ here
cannot be improved.

Here is the plan of the rest of the paper. In the next section we will
review a well-known theorem of Besicovitch (sometimes called Besicovitch lemma).
This theorem implies that for each $2$-disc of area $A$ with a boundary of length $L$ there exist two points on the boundary
such that the distance between these points in the disc does not exceed $\sqrt{A}$ but the length of a shortest
arc of the boundary connecting these points is at least ${L\over 4}$. This theorem implies that if $L>2\sqrt{A}$
then one can subdivide the disc into two subdiscs with smaller areas and shorter boundaries by a minimizing geodesic connecting
these points. Below we will call this minimizing geodesic {a Besicovitch cut} (see Fig. 1).

\begin{figure}[center]
\includegraphics[scale=0.4]{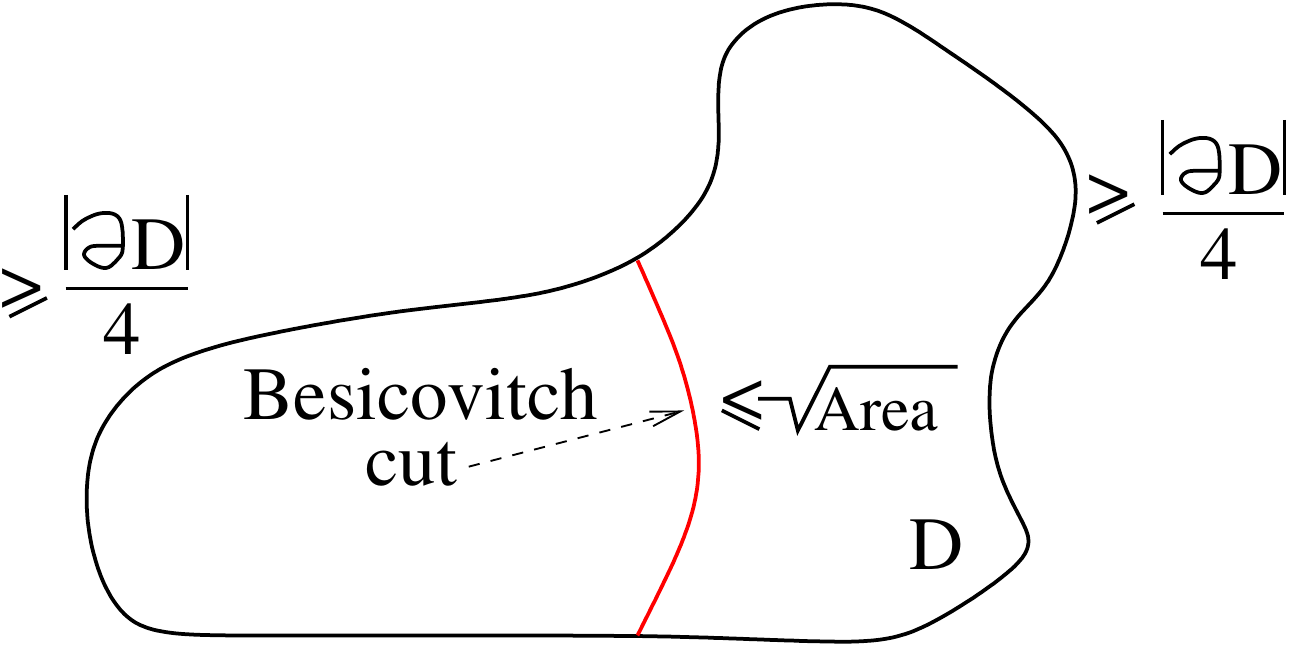}
\caption{Besicovitch lemma}
\label{lnr1}
\end{figure}

Then we will use this theorem to reduce the estimates of Theorem 1.6 A-C
for a disc $D$ to slightly stronger estimates
for all subdiscs $D'$ of $D$
such that the length of the boundary of $D'$ does not exceed
$6\sqrt{Area(D)}$. The idea here is to subdivide the original disc into smaller and smaller
subdiscs with shorter and shorter boundaries until their lengths will 
become less than $6\sqrt{Area(D)}$. In the same section we apply this result to prove
the desired assertion for subdiscs of $D$ with areas not exceeding a very small constant. (The idea is that this result becomes almost obvious if the length
of the boundary is also very small).

At the beginning of section 3 we will review a result by P. Papasoglu ([P])
asserting that for every Riemannian $2$-sphere $S$ and every $\epsilon$ there
exists a simple closed curve of length $\leq 2\sqrt{3}\sqrt{Area(S)}+\epsilon$ that
divides the sphere into two domains with areas at least ${1\over 4}Area(S)$
 and at most ${3\over 4}Area(S)$.  Then we will
prove an analogous result for Riemannian $2$-discs. Section 4 contains
two auxilliary results about a relationship of $d_D$ and $d_{D'}$ for
a subdisc $D'$ of $D$. One of these results (Lemma 4.1) asserts that if the distance 
between
a point $p\in\partial D$ and and a point $p'\in\partial D'$ does not exceed $l$
then $$d_{D'}+l\leq d_D +\vert\partial D'\vert,\ \ \ \ (*)$$
see Fig. 2.

\begin{figure}[center]
\includegraphics[scale=0.3]{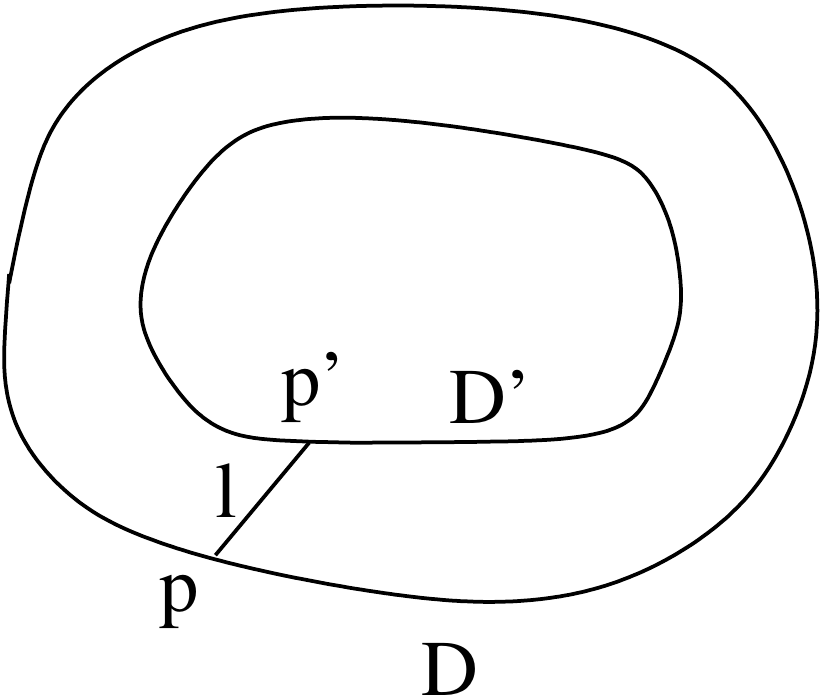}
\caption{}
\label{lnr2}
\end{figure}

This property is easy to establish, but as we
will see, it is very helpful in the proof of our main result, and is the main reason why
$d_D$ in our estimates suites us better than the diameter of $D$.

Section 5 contains the proof of 
Theorem  1.6 A-C (and, thus, Theorem 1.1). As we have already mentioned the proof is
inductive. On each step of induction we increase the area of subdiscs $D'$ of $D$ by a factor
of ${4\over 3}$ and prove that these new subdiscs $D'$ still satsfy the inequalities of Theorem 1.6.
At this point we already know that it is sufficient to prove a slightly stronger
inequality for subdiscs $D''$ of $D'$  with
boundary of length not exceeding $6\sqrt{Area(D')}$. The idea is to apply a version of the quoted result
of Papasoglu for discs to $D''$. We obtain a disc $\bar D\subset D''$ of area $\leq {3\over 4}Area(D'')\leq {3\over 4}Area (D')$.
In order to construct a path homotopy between an arc $l_1$ of $\partial D''$ connecting some two points $p,q\in\partial D''$
with its complementary arc $l_1$ we proceed as follows.
Connect $p$ and $q$ with the closest points $p_1, q_1$ of $\partial \bar D$ by minimizing geodesics
that we denote as $\alpha_1$ and $\alpha_2$. We would like to first homotope $l_1$ to the join of arcs $\alpha_1, \gamma_1$ and $\alpha_2$,
where $\gamma_1$ denotes one of the two arcs of $\partial\bar D$ connecting $p_1$ and $p_2$, then homotope $\gamma_1$ into
its complement $\gamma_2$ in $\partial\bar D$ while keeping $\alpha_1$ and $\alpha_2$ intact, and finally homotope
the join of $\alpha_1,\gamma_2$ and $\alpha_2$ into $l_2$ (see Fig. 3).

\begin{figure}[center]
\includegraphics[scale=0.3]{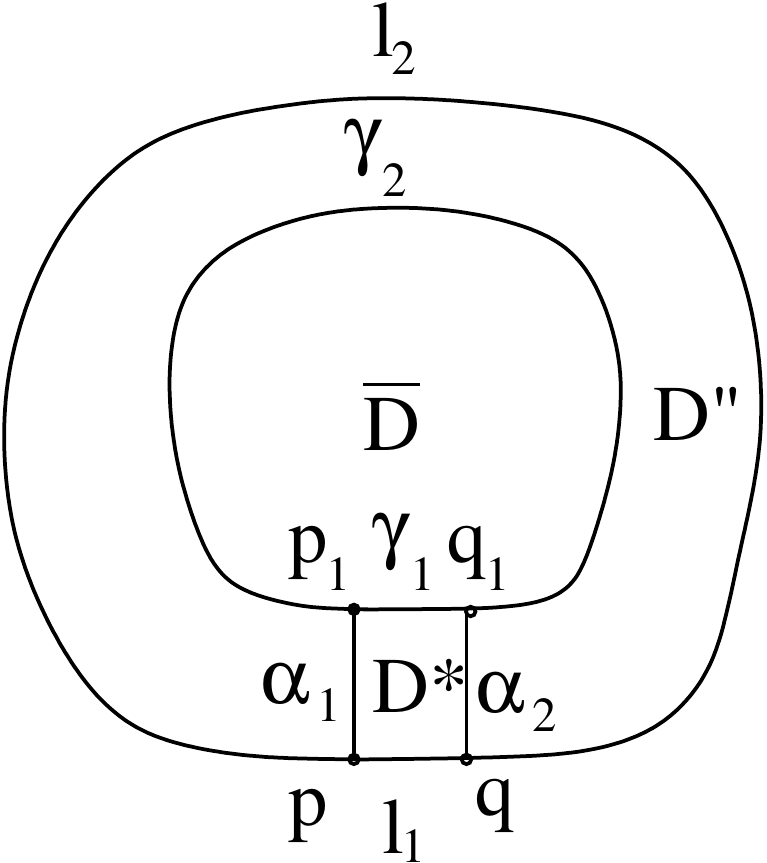}
\caption{}
\label{lnr3}
\end{figure}

The basic idea here is that all these three homotopies
are between complementary arcs of subdiscs of $D''$ with areas $\leq {3\over 4}Area(D')$, and therefore the induction
assumption applies. Yet a cursory examination of this idea reveals some potential flaws. First, consider the second
step of these construction. We are going to keep $\alpha_1$ and $\alpha_2$
constant, and will be looking for a path homotopy between $\gamma_1$ and
$\gamma_2$ in $\bar D$. The induction assumption leads to an upper bound for lengths of curves in this path homotopy
that contains a term equal to $2d_{\bar D}$. Note that $d_{\bar D}$ can be as large as $d_{D''}$ or even larger. On top of that
we do not have an upper bound for lengths of either of geodesics $\alpha_i$ that
is better than $d_{D''}$. Thus, we could end up
with an upper bound involving a term equal to $4d_{D''}$ instead of $2d_{D''}$. But it is here that we are saved
by the property $(*)$. This property implies that $2d_{\bar D}+length(\alpha_1)+length (\alpha_2)$ can be majorized by $2d_{D''}+
2\vert\partial\bar D\vert$. On the other hand $\vert\partial\bar D\vert$ can be majorized in terms of $\sqrt{Area(D'')}$. We do not mind 
extra terms of the form $const\ \sqrt{Area(D')}$, since when we apply the induction assumption to subdiscs of $D''$, their areas do
not exceed ${3\over 4}Area(D')$. Therefore
we have large ``savings" ($\sim 91.91\sqrt{Area(D')}$) because $686\sqrt{{3\over 4} Area (D')}$ is much less than the allowed term
$686\sqrt{Area(D')}$. We can use the difference to compensate for extra
terms of the
form $const\sqrt{Area(D')}$ that appear as the result of our constructions.
(As it turns out, the total contribution of such terms is majorized by
$(78+8\sqrt{3})\sqrt{Area(D')}\sim 91.86\sqrt{Area(D')}$.)
\par
Yet the second difficulty with our basic idea is more serious. Consider the very first step of the suggested three-step construction
of a desired path homotopy. We are looking for a path homotopy between $l_1$ and $\alpha_1*\gamma_1*\alpha_2$ in the subdisc $D^*$ of $D''$
bounded by these two arcs. The upper bound that we obtain from the induction assumption will involve the summand $\vert\partial D^*\vert+2d_{D^*}$.
The length of $\vert\partial D^*\vert$ is greater than 
$length(\alpha_1)+length(\alpha_2)$, and our best upper bound 
for $d_{D^*}$ is
$d_{D''}+length(\alpha_1)+length(\alpha_2)+length (\gamma_1)$ (Lemma 4.2). As 
$length(\alpha_1)+length(\alpha_2)$ can be almost as large as $2d_{D''}$, our upper bound for the expression $\vert\partial D^*\vert +2d_{D^*}$ will contain
the summand $6d_{D''}$, which is significantly
worse than $2d_{D''}$ in the required upper bound and is, therefore,
unsuitable for our purposes.
Of course, this situation will be problematic only if the length of either $\alpha_1$ or $\alpha_2$ is
much larger
than $\sqrt{Area(D')}$ (as in the opposite case we will obtain an upper
bound of the form $2d_{D''}+const\sqrt{Area(D')}$, which is acceptable for us).
It is easy to see that in this case both $\alpha_1$ and $\alpha_2$ 
will be very long
in comparison with $\sqrt{Area(D')}$. Now our key observation is that in this case an application of Besicovitch lemma will produce a Besicovitch cut that
will connect a point on $\alpha_1$ with a point on $\alpha_2$ (Lemma 5.2). 
This cut divides $D^*$ into two subdiscs $D_1^*$ and $D_2^*$ such that 
their boundaries are significantly shorter than the boundary of $D^*$ (see Fig. 4).

\begin{figure}[center]
\includegraphics[scale=0.3]{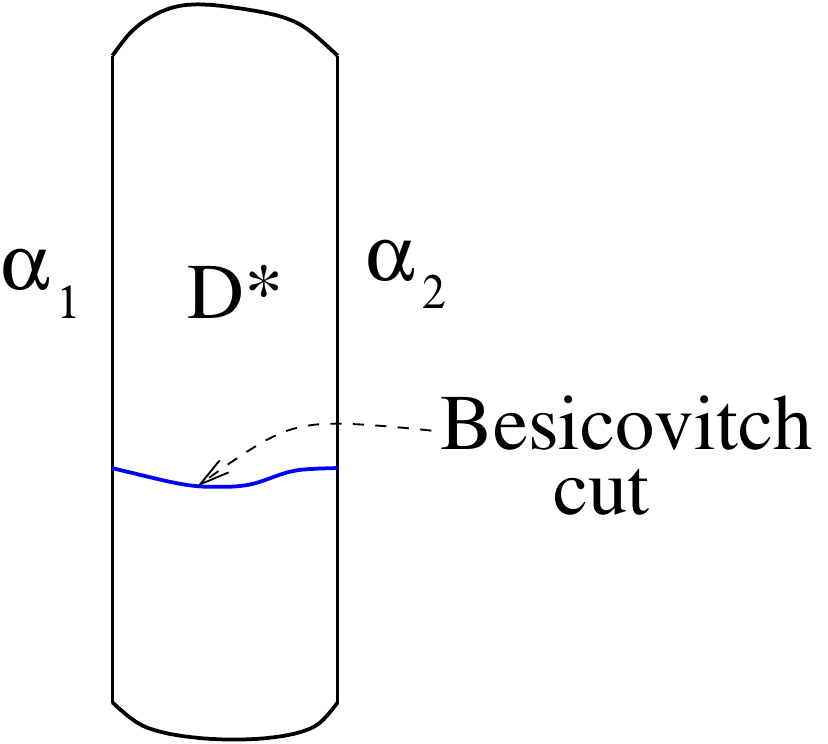}
\caption{ If the lengths of $\alpha_1,\alpha_2>>\sqrt{Area(D'')}>\sqrt{Area(D^*)}$ are much greater than lengths of two other sides
then each Besicovitch cut connects $\alpha_1$ and $\alpha_2$.}
\label{lnr4}
\end{figure}
 
If the lengths of pairs of sides inherited from $\alpha_i$
are still large in comparison with $\sqrt{Area(D')}$ for either of these discs, then  we will apply the Besicovitch lemma to this disc again. And so on. Eventually we will end up with a partition
of $D^*$ into a stack of discs with boundaries of lengths
commensurable with $\sqrt{Area(D')}$ (see Fig. 5). 

\begin{figure}[center]
\includegraphics[scale=0.2]{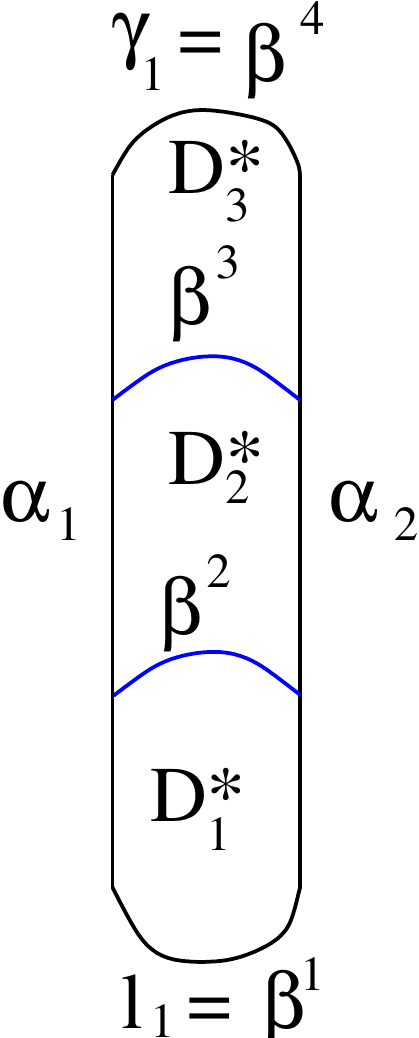}
\caption{Stack of discs partitioning $D^*$.}
\label{lnr5}
\end{figure}

Denote the number of discs in this stack by $M$. Then we will start constructing our path homotopy from
$l_1$ to $\alpha_1*\gamma_1*\alpha_2$. On each of $M$ steps
we will be constructing a path homotopy through one of the discs in the stack. On the $k$th step, ($k=1,\ldots, M$), the homotopy will pass from
$\alpha_1^{k}*\beta^k*\alpha_2^{k}$ to $\alpha_1^{k+1}*\beta^{k+1}*\alpha_2^{k+1}$ through $D^*_k$, where $D^*_k$ denotes the $k$th disc in the stack,
$\beta^k$ and $\beta^{k+1}$ are Besicovitch cuts forming the
``bottom" and ``top"
sides of $D^*_k$ (with the exception of $\beta^1$ and $\beta^M$.
$\beta^1=l_1$ and $\beta^M=\gamma_1$.)
Further, let $\alpha_i^k$ denote initial segments of $\alpha_i$ up to the endpoints of $\beta^k$. 
Again the quoted property $(*)$ of $d_{D^*_k}$ saves us from an undesirable accumulation of lengths of $\alpha_i^k$  with
$d_{D^*_k}$, when we use the induction assumption to bound the lengths of curves in all of these path homotopies.
\par
In section 7 we will
deduce Theorem 1.2 from Theorem 1.1.
Here the key intermediate result (Theorem 6.1)
is that an arbitrary Riemannian $2$-disc $D$ can be subdivided into two subdiscs
with areas in the interval $[{1\over 3}Area(D)-\epsilon^2, {2\over 3} Area(D)
+\epsilon^2]$ by a simple curve of length $\leq 2diam(D)+2\epsilon$
connecting two points of $\partial D$, where $\epsilon$ can
be made arbitrarily small. 
The proof of this result will be given in section 6. It uses
a modification of Gromov's filling technique and is reminiscent of a proof
of a version of the result of Papasoglu quoted above presented
by F. Balacheff and S. Sabourau in [BS]. At the end of section 6 we demonstrate
that for each $a<{1\over 2}$ there exists a constant $C(a)$ such that
each Riemannian $2$-sphere of area $A$ and diameter $d$ can be divided into two discs of
area $\geq aA$ by a simple closed curve of length $\leq C(a)d$.
This theorem contrasts with the recent result of one of the authors ([L2]),
who proved that for every constant $C$ there exists a Riemannian sphere of diameter $1$
that cannot be divided into two subdomains of equal area by any $1$-cycle
of length $\leq C$.

In the case, when $\sqrt{Area(D)}>>diam(D)$, a repeated application of Theorem 6.1 enables us to break the disc $D$
into smaller and smaller discs $D_i$ until $\sqrt{Area(D_i)}$ becomes much smaller
than $d_{D_i}$. However, the lengths of boundaries of $D_i$ might
increase in the process by $const\ diam(D)\ln{\sqrt{Area(D)}\over diam(D)}$.
Now the idea is to apply
Theorem 1.6 to the resulting discs $D_i$ (and to construct the desired path
homotopy as a combination of the obtained path homotopies).
But first we will need
to prove Theorem 1.6 D. Observe that an important difference between
the upper bounds for $pdias$ in Theorems 1.6. C and Theorem 1.6 D is
the following. If we consider the estimates for $pdias$ as a function of 
$\vert \partial D\vert$ for some fixed values of $Area(D)$ and $diam(D)$,
then the estimate of Theorem 1.6. D looks like $\vert\partial D\vert +O(1)$,
but the estimate of Theorem 1.6. C contains also a logarithmic term
and looks like $\vert \partial D\vert
+ 2\log_{4\over 3}\vert\partial D\vert +O(1)$. This additional logarithmic term
means that the estimate of Theorem 1.6. C cannot be used
to prove the existence of a finite $exc(d,D)$.
To prove Theorem 1.6 D we will be using a trick that reduces
constructing of a path homotopy between two complementary arcs of $\partial D$
to constructing of path homotopies between sides of
digons formed by two minimizing geodesics between pairs of
points of $\partial D$ (Proposition 7.3). The key point is that
the length of $\partial D$ does not enter into the estimates for areas,
diameters and perimeters of such digons. Note that the diameter of any such
minimizing digon $A$ does not exceed $diam(D)$, but, in general, it is not 
true that $d_A\leq d_D$ (although it is easy to see that $d_A\leq {3\over 2}d_D$). Because of this reason we did not apply this trick on earlier stages of our proof.

At the end of section 7
we will explain how Theorem 1.3 follows from Theorems 1.1 and 1.2.
In section 8 we present proofs of Theorem AA and Theorem BB asserting that the upper bounds for $exc(d,A)$ and $dias(M)$
provided by Theorems 1.3 and 1.6 are optimal up to a constant factor.
Section 9 contains applications to upper bounds for lengths of
geodesics on Riemannian $2$-spheres.
We prove the upper bounds for lengths of different geodesics between two fixed
points
on a Riemannian $2$-sphere that depend on the diameter and the area of
the $2$-sphere. We also review applications of the results of
the present paper to upper bounds for the lengths of three
shortest simple periodic geodesics that depend only on the area and the diameter given in our recent paper [LNR2].
These estimates are better than the upper bounds that depend only on the diameter proven in [NR1], [NR2] and [R]  in the case when the ratio of the area
to the square of the diameter is very small. In fact, we believe that in this case our estimates for lengths of two of these geodesics are nearly optimal.
 
\section{Besicovitch Lemma and reduction to the case of curves with short boundaries}

The main tool of this paper is the following theorem (``Besicovitch lemma")
due to A.S. Besicovitch \cite{B} (see also
\cite{BBI} and \cite{Gr99} for generalizations and many applications of this
theorem).

\begin{theorem} \label{besicovitch}
Let $D$ be a Riemannian 2-disc. Consider a subdivision of $\partial D$ into four consecutive subarcs (with disjoint interiors)
$\partial D =a \cup b \cup c \cup d$.
Let $l_1$ denote the length of a minimizing geodesic between $a$ and $c$; $l_2$ denote the length of a minimizing geodesic between $b$ and $d$.
Then $$Area(D) \geq |l_1| |l_2|$$
\end{theorem}

{\bf Corollary 2.1.A.} {\it Let $D$ be a $2$-disc. There exist two points $p,q\in\partial D$ such that $dist_D(p,q)\leq\sqrt{Area(D)}$ and $dist_{\partial D}(p,q)\geq {\vert\partial D\vert\over 4}$. Therefore, if $\vert\partial D\vert > 2\sqrt{Area(D)}$
then a minimizing geodesic connecting $p$ and $q$ in $D$ divides $D$ into two subdiscs with smaller perimeters (see Fig. 1).}

{\bf Proof.} Subdivide the boundary of $D$ into four subarcs of equal length ($={\vert\partial D\vert\over 4}$) and apply the previous theorem.

\medskip
In this section we will use Besicovitch lemma to prove two lemmae.
Lemma \ref{boundary} implies
that the second inequality of Theorem \ref{main'} follows from the first.
Lemma \ref{small area} says that boundaries of small subdiscs of $D$ can be contracted through short curves.
The proof of Lemma \ref{boundary} uses the idea of Corollary 2.1.A: We inductively divide the disc into discs with smaller areas and
considerably smaller perimeters. In order to make perimeters considerably smaller we need to assume that $\vert D\vert\geq const\sqrt{Area(D)}$,
where $const > 2$. In fact, it is convenient for us to takea $const=6$ here. After each application of the Besicovitch lemma the perimiters of the discs
drop at least by a multiplicative factor $<1$; this happens until their perimeters remain $>6\sqrt{Area(D)}$. This enables one to reduce
the construction of a path homotopy between an arc of $\partial D$ and its complement to constructing similar path homotopies for subdiscs of $D$
with perimeters $\leq 6\sqrt{Area(D)}$.

\begin{lemma} \label{boundary}
\textbf{(Reduction to a Short Boundary Case)}
Let $\epsilon_0, C$ be any non-negative real numbers.
\par\noindent
A. Suppose that $|\partial D| > 6 \sqrt{Area(D)}$ and that for all subdiscs
$D' \subset D$ satisfying $|\partial D'| \leq 6 \sqrt{Area(D)}$
we have $pdias(D',D)\leq (1+\epsilon_0)|\partial D'| + C \sqrt{Area(D)} + 2d_{D'}$. Then
$$ pdias(D) \leq (1+\epsilon_0) |\partial D|
+ 2\lceil \log_{\frac{4}{3}}(\frac{|\partial D|-4\sqrt{Area(D)}}{2\sqrt{Area(D)}})  \rceil \sqrt{Area(D)}
+ C \sqrt{Area(D)} + 2d_{D}.$$
\par\noindent
B. Assume that $D$ is contained in a disc $D_0$, and all subdiscs
$D' \subset D$ satisfying $|\partial D'| \leq 6 \sqrt{Area(D)}$
satisfy $pdias(D',D_0)\leq (1+\epsilon_0)|\partial D'| + C \sqrt{Area(D)} + 2d_{D'}$. Then
$$ pdias(D, D_0) \leq (1+\epsilon_0) |\partial D|
+2\lceil \log_{\frac{4}{3}}(\frac{|\partial D|-4\sqrt{Area(D)}}{2\sqrt{Area(D)}})  \rceil \sqrt{Area(D)}
+ C \sqrt{Area(D)} + 2d_{D}.$$

\end{lemma}

\begin{proof}
A. First, we are going to prove A.
 For each subdisc $D' \subset D$ define

$$n(D')=\log_{\frac{4}{3}}(\frac
    {|\partial D'| - 4 \sqrt{Area(D)}}
    {2 \sqrt{Area(D)}})$$

For each $n \in \{0,..., \lceil n(D) \rceil \}$
(where $\lceil$ x $\rceil$ denotes the integer part of x+1)
and every subdisc $D' \subset D$ with
$n-1 < n(D') \leq n$ we will show that
$pdias(D',D) \leq (1+\epsilon_0)|\partial D'| + 2n \sqrt{Area(D)} + C \sqrt{Area(D)} + 2d_{D'}$

For $n=0$ we have $|\partial D'| \leq 6 \sqrt{Area(D)}$
so we are done by assumption in the statement of the theorem.

Suppose the conclusion is true for all integers smaller than $n$.
Let $p,q\in \partial D'$. Let $l_1$ and $l_2$ be two subarcs of $\partial D'$ from $p$ to $q$, $|l_2|\leq|l_1|$.
We will construct a homotopy of paths from $l_1$ to $l_2$ of length
$\leq (1+\epsilon_0)(|l_1|+|l_2|)+(C + 2n) \sqrt{Area(D')}+2 d_{D'}$.

Subdivide $l_1 \cup -l_2$ into four arcs $a_1$, $a_2$, $a_3$ and $a_4$ of
equal length  so that the center
of $a_2$ coincides with the center of $l_2$.
By Besicovitch lemma there exists a curve $\alpha$ between opposite sides $a_1$ and $a_3$ or
$a_2$ and $a_4$ of length $\leq \sqrt{Area(D')}$.

We have two cases (see Fig. 6).

\begin{figure}[center]
\includegraphics[scale=0.5]{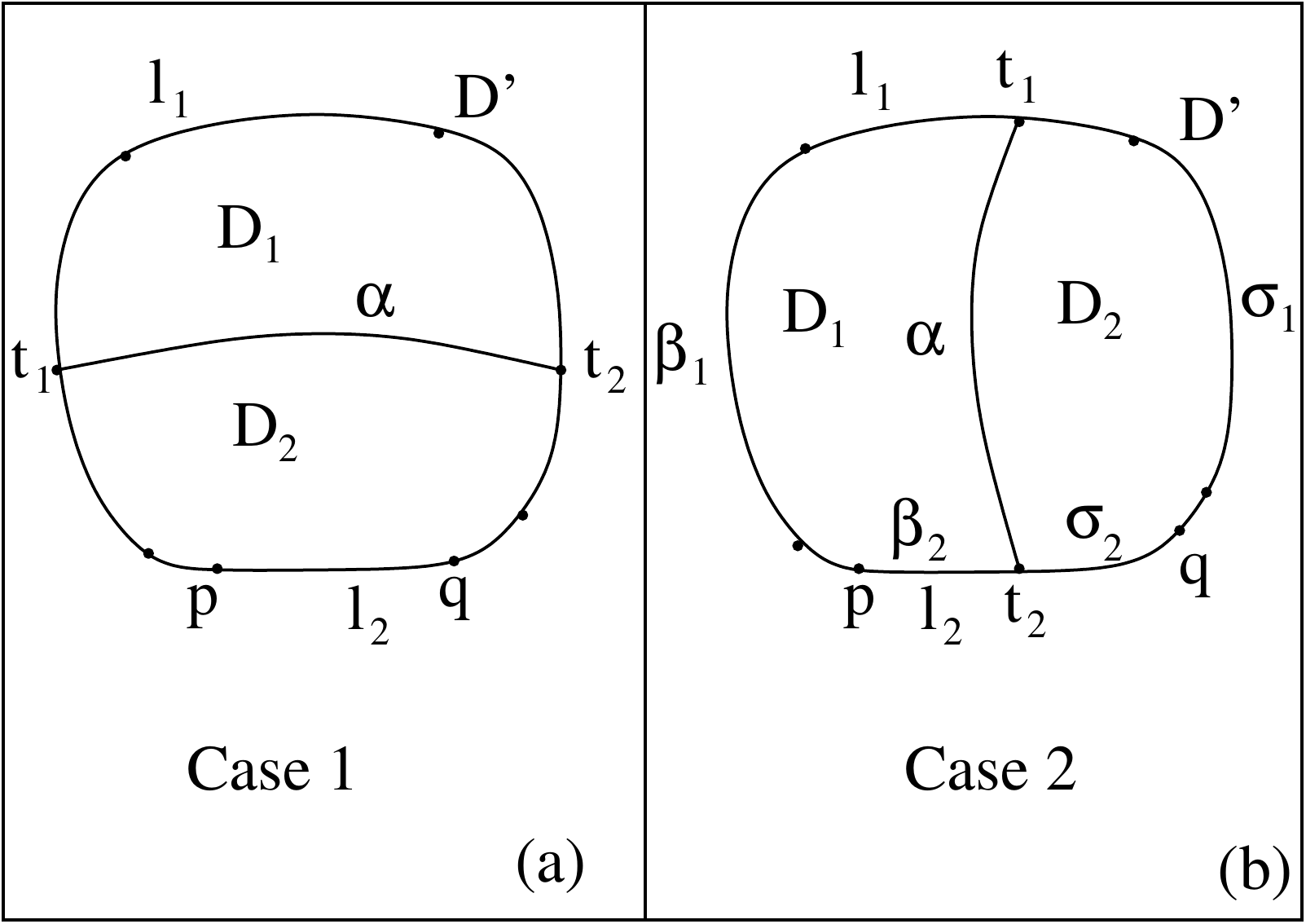}
\caption{}
\label{lnr6}
\end{figure}

\textbf{Case 1.} Both endpoints $t_1$ and $t_2$ of $\alpha$ belong to the
same arc $l_i$ ($i = 1$ or $2$).
Denote the arc of $l_i$ between $t_1$ and $t_2$ by $\beta$.
Note that $\frac{1}{4} (|l_1|+|l_2| )\leq |\beta| \leq \frac{3}{4} (|l_1|+|l_2|)$.
In particular, the disc $D_1$ bounded by $\alpha \cup -\beta$ has boundary of length
$\leq \frac{3}{4} |\partial D'| + \sqrt{Area(D')} \leq (4 + 2(\frac{4}{3})^{n-1}) \sqrt{Area(D)}$.
The induction assumption  implies that
$pdias(D_1,D) \leq (1+\epsilon_0)|\partial D_1| + (C + 2n -2) \sqrt{Area(D)} + 2d_{D_1}$.

We claim that $d_{D_1} \leq d_{D'} + \frac{1}{2} \sqrt{Area(D')}$. Indeed, let $y \in \partial D_1$.
If $y \in \partial D'$ then the geodesic from $y$ to $x$ does not cross $\alpha$ as both are minimizing
geodesics, hence the distance in $D_1$ $d_{D_1}(y,x)\leq d_{D'}$.
If $x \in \alpha$ then the triangle inequality implies that $d_{D_1}(y,x)\leq \frac{1}{2}|\alpha|+d_{D'} $.

Hence, for an arbitrarily small $\delta >0$ we can homotop $l_i$ to
$pt_1 \cup \alpha \cup t_2q$ through curves of length
$$ \leq |l_i \setminus \beta| + (1+\epsilon_0)|\partial D_1| + (C + 2n -2) \sqrt{Area(D)} + 2d_{D_1} + \delta$$
$$ \leq (1+\epsilon_0)|\partial D'|+ \sqrt{Area(D)} + (C + 2n -2) \sqrt{Area(D)} + 2d_{D} + \sqrt{Area(D)} + \delta .$$

Now consider the disc $D_2$ bounded by $pt_1 \cup \alpha \cup t_2q \cup - l_j$, where $l_j$ ($j \neq i$)
is the other arc. As in the case of $D_1$, we can homotop $pt_1 \cup \alpha \cup t_2q$ to $l_j$ through curves of length
$ \leq (1+\epsilon_0)|\partial D'| + 2n \sqrt{Area(D)} + C \sqrt{Area(D)} + 2d_{D'}.$

\textbf{Case 2.}
$t_1 \in l_1$, $t_2 \in l_2$.
Let $\beta_i$ denote the subarc of $l_i$ from $p$ to $t_i$ and $\sigma_i$ denote the subarc of $l_i$
from $t_i$ to $q$. Consider the subdisc $D_1 \subset D'$ bounded by $\beta_1 \cup \alpha \cup -\beta_2$.
As in Case 1 the inequality $|\partial D_1| \leq \frac{3}{4} |\partial D'| + \sqrt{Area(D')}$
combined with the induction assumption implies that
$pdias(D_1,D) \leq (1+\epsilon_0)|\partial D_1| + (C + 2n -2) \sqrt{Area(D)} + 2d_{D_1}$.
Using the estimate $d_{D_1} \leq d_{D'} + \frac{1}{2} \sqrt{Area(D')}$ we can homotop $l_1$ to
$\beta_2 \cup -\alpha \cup \sigma_1$ through curves of length
$$\leq (1+\epsilon_0)|\partial D'| + (2n + C) \sqrt{Area(D)} + 2d_{D'} +\delta. $$

In exactly the same way we homotop $\beta_2 \cup -\alpha \cup \sigma_1$ to $l_2$ using the inductive assumption for the other disc $D_2=D'\setminus D_1$.
%
%

This proves that $pdias(D) \leq (1+\epsilon_0)|\partial D| + 2 \lceil n(D) \rceil
 \sqrt{Area(D)} + C \sqrt{Area(D)} + 2d_{D}$.

This completes the proof of A. The proof of its
relative version B is almost identical to the proof of A.

\end{proof}

\begin{lemma} \label{small area}
\textbf{(Small Area)} Given a positive $\epsilon_0$ there exists a positive $\epsilon$, such that if $D \subset D_0$ with $Area(D) < \epsilon,$
then
$$pdias(D,D_0)\leq (1+\epsilon_0)|\partial D|,$$
when $|\partial D|\leq 6\sqrt\epsilon,$ and
$$ pdias(D,D_0) \leq (1+\epsilon_0)|\partial D|
+2\lceil \log_{\frac{3}{4}}(\frac{|\partial D|-4\sqrt{Area(D)}}{2\sqrt{Area(D)}})  \rceil \sqrt{Area(D)}
+ 2d_{D},$$
when $|\partial D|>6\sqrt\epsilon.$
\end{lemma}
 
\begin{proof}
Lemma \ref{boundary} B implies that in order to prove the
second inequality it is enough to find $\epsilon>0$ such
that for all subdiscs $D'$
of $D$
with $|\partial D'|\leq 6 \sqrt{\epsilon}$
$$pdias(D', D_0)\leq (1+\epsilon_0)\vert\partial D'\vert
+2d_{D'}.$$
For all sufficiently small radii $r$ every ball $B_r(p) \subset D$ is bilipschitz
homeomorphic to a convex subset of the positive half-plane $\mathbb{R}^2_+$
with bilipshitz constant $L=1+O(r^2)$.

Hence, for a sufficiently small $\epsilon$ if $|\partial D'| \leq 6 \sqrt{\epsilon}$,
then $pdias(D',D_0) \leq (1+O(\epsilon))pdias(U,V)$, where $U \subset V \subset \mathbb{R}^2_+$,
$|\partial U| \leq (1+O(\epsilon))|\partial D'|$ and $V$ is convex.
We will show that $pdias(U,V) \leq |\partial U|$
thereby proving the result.

Let $p,q \in \partial U$ and $l_1:[0,1] \rightarrow V$, $l_2:[0,1] \rightarrow V$  
be two arcs of $\partial U$ from $p$ to $q$. Let $\alpha^i_t:[0,1] \rightarrow V$
denote a parametrized straight line from $p$ to $l_i(t)$.
We define a homotopy of paths from $l_1$ to $l_2$ as $\gamma_t = \alpha^1 _{2t} \cup l_1 |_{[2t,1]}$
for $0 \leq t \leq \frac{1}{2}$ and $\gamma_t = \alpha^2 _{2-2t} \cup l_2 |_{[2-2t,1]}$.
We have $|\gamma_t| \leq max(|l_1|,|l_2|) \leq |\partial U|.$
\par
Now we can choose $\epsilon>0$ so that
$(1+O(\epsilon))pdias(U,V)\leq (1+\epsilon_0)pdiast(U,V)$,
and the desired assertion follows.
\end{proof}
{\bf Remark.} Note that it is not difficult to prove the existence of $\epsilon>0$ such that
for each disc
$D\subset D_0$ of area $\leq\epsilon$
one has $pdias(D, D_0)\leq \vert \partial D\vert$. 
Yet the
proof is more complicated than the proof above. Moreover, this strengthening
of Lemma 2.3 does not lead to any improvements of our main estimates.
Therefore, we decided to state Lemma 2.3 only in its weaker form.

\section{Subdivision by short curves}

The following theorem was proven 
by P. Papasoglu in \cite{P}. For the sake of completeness we will present
a proof which is a slightly simplified version of the
proof given by Papasoglu.

\begin{theorem} \label{sphere subdivision}
\textbf{(Sphere Subdivision)}
Let $M=(S^2,g)$ be a Riemannian sphere. For every $\delta>0$
there exists a simple closed curve $\gamma$ subdividing $M$ into two discs $D_1$ and $D_2$, such that
$\frac{1}{4} Area(M) \leq Area(D_i) \leq \frac{3}{4} Area(M)$ and $|\gamma | \leq 2 \sqrt{3} \sqrt{Area(M)}+\delta$
\end{theorem}

\begin{proof}
Consider the set $S$ of all simple closed curves on M dividing M into two subdiscs each of area $\geq {1\over 4} Area(M)$.
To see that this set is non-empty one can take a level set of a Morse function
on $M$ and connect its components by geodesics. From arcs of these geodesics one can obtain
paths between components of the level set that can be made disjoint by a small perturbation.
Traversing each of the connecting paths twice one obtains a closed curve that becomes simple after a small perturbation.

Choose a positive $\epsilon$.
Let $\gamma \in S$ be a curve that is $\epsilon -$minimal. (In other words,
its length is greater than or equal to $\inf_{\tau\in S}\vert\tau\vert+\epsilon$.)
Let $D$ be one of the
two discs forming $M \setminus \gamma$ that has
area $\geq  \frac{1}{2} Area(M)$.
If we subdivide $\gamma$ into four equal arcs then by Besicovitch Lemma there is
a curve $\alpha$ connecting two opposite arcs of length $\leq {\sqrt{3}\over 2} \sqrt{A}$. Observe that
$\alpha$ subdivides $D$ into two discs, and at least  one of these discs has
area $\geq {1\over 4} Area(M)$.
Hence, the boundary of this disc is an element of $S$ of length
$\leq {3\over 4} |\gamma| + |\alpha|$. By $\epsilon -$minimality of $\gamma$ we must have
$$|\gamma| \leq {3\over 4} |\gamma| + {\sqrt{3}\over 2} \sqrt{A} + \epsilon.$$
Therefore, $|\gamma| \leq 2 \sqrt{3}  \sqrt{A} + 4\epsilon$.
\end{proof}

Our next result is an analog of the previous result for $2$-discs.

\begin{proposition} \label{disc subdivision}
\textbf{(Disc Subdivision Lemma)}
Let $D$ be a Riemannian 2-disc. For any $\delta>0$
there exists a subdisc $\overline{D} \subset D$ satisfying

(1) $\dfrac{1}{4} Area(D) - \delta^2 \leq Area(\overline{D}) \leq \dfrac{3}{4} Area(D) + \delta^2$

(2) $|\partial \overline{D} \setminus \partial D | \leq 2 \sqrt{3} \sqrt{Area(D)} + \delta$

\end{proposition}

\begin{proof}
Without any loss of generality we can assume
$\delta\leq{\sqrt{Area(D)}\over 10}$.
Attach a disc $D'$ of area $\leq\delta^2$ to the boundary of $D$ so that $M= D' \cup D$
is a sphere of area $\leq Area(D) + \delta^2$. We apply Theorem \ref{sphere subdivision}
to $M$ to obtain a close curve $\gamma$ of length $\leq 2\sqrt{3}\sqrt{Area(D)}+\delta$ that divides $D$ into two subdiscs $D_1$ and $D_2$ with areas in the interval
$[{1\over 4}Area(D)-\delta^2, {3\over 4} Area(D)+\delta^2]$. Without any loss
of generality we can assume that either $\gamma$ does not intersect
$\vert\partial D\vert$ or intersects it transversally.
(Note that the idea of attaching a disc
of a very small area to the boundary of $D$ and applying  Theorem 3.1 appears
in \cite{BS}.)

If $\gamma \cap \partial D$ is empty then $D_i \subset D$ for one of $D_i$'s and setting $\overline{D} = D_i$
we obtain the desired result.

A more difficult case arises when $\gamma \cap \partial D \neq \oslash$.
For each $i=1,2$ $D_i \cap D$ may have several connected components.
Those components, $D^j$, are subdiscs of $D$ of area $\leq {3\over 4}Area(D)+\delta^2$. If the area of one of them is $\geq {1\over 4}Area(D)-\delta$, then
we can choose this subdisc as $\overline{D}$, and we are done. Otherwise,
we can start erasing connected components of $\gamma\bigcap D$ one by one.
When we erase a connected component of $\gamma\bigcap D$, the two subdiscs adjacent
to the erased arc merge into a larger subdisc of area $\leq {1\over 2}Area(A)-2\delta^2$. We continue
this process until we obtain a new subdisc of area $\geq {1\over 4}Area(A)-\delta^2$, and choose this subdisc as $\overline{D}$.
%
%
%
%
\end{proof}
{\bf Remark.} Theorem 3.1 naturally leads
to the following
question: For which values of $a\in ({1\over 4}, {1\over 2})$ does there exist
a constant $c(a)$ such that every Riemannian $2$-sphere of area $A$
can be subdivided
into two discs of area $> aA$ by a simple closed curve of length $\leq c(a)\sqrt{A}$? It is not difficult to see that $a$ cannot exceed ${1\over 3}$. Indeed,
consider a three-legged starfish made out of three congruent halves
of very thin ellipsoids of revolution. The length of each of these halves
of ellipsoids of revolution is much larger than the square root of its area,
and the area of its subset formed by all points
at the distance $<const \sqrt{A}$ from the line,
where it meets the other two pieces, is infinitesimally
small. Using these remarks
it is not difficult to see that no matter how
one places a simple closed curve of length $\leq const \sqrt{A}$ on the
three-legged starfish, the area
of the smaller one of the two discs bounded by this curve can be at most $({1\over 3}+o(1))A$. On the other hand, one of the authors (E.L.) recently proved
the existence of an absolute constant $C$
such that each Riemannian $2$-sphere of area $A$ can be
subdivided into two discs of area $\geq {1\over 3}A$ by a simple closed
curve of length $\leq C\sqrt{A}$ ([L2]).

\section{Bounds for $d_{D'}$.}

We will also need the following lemmae relating $d_D$ with $d_{D'}$ for
a subdisc $D'$ of $D$.

\begin{lemma} \label{diameter}
Let $D'$ be a subdisc of $D$, $p \in \partial D$ and $p' \in \partial D'$
be two points connected by a minimizing geodesic $\alpha$ in $D$. Then $d_{D'}+\vert\alpha\vert \leq d_D + |\partial D'|$
\end{lemma}

\begin{proof} 
Let $\beta$ be a minimizing geodesic in $D'$ from a point on the boundary of $D'$ to a point $x \in D'$,
s.t. $|\beta|= d_{D'}$ (It exists by compactness). Let $\gamma$ be a minimizing geodesic from $p$ to $x$.
Denote by $\gamma_1$ the arc of $\gamma$ from $p$ to a point $u_2$ where it first intersects $\partial D'$
and by $\gamma_2$ the arc from a point $u_1$ where it last intersects $\partial D'$ to $x$ (see Fig. 7).

\begin{figure}[center]
\includegraphics[scale=0.5]{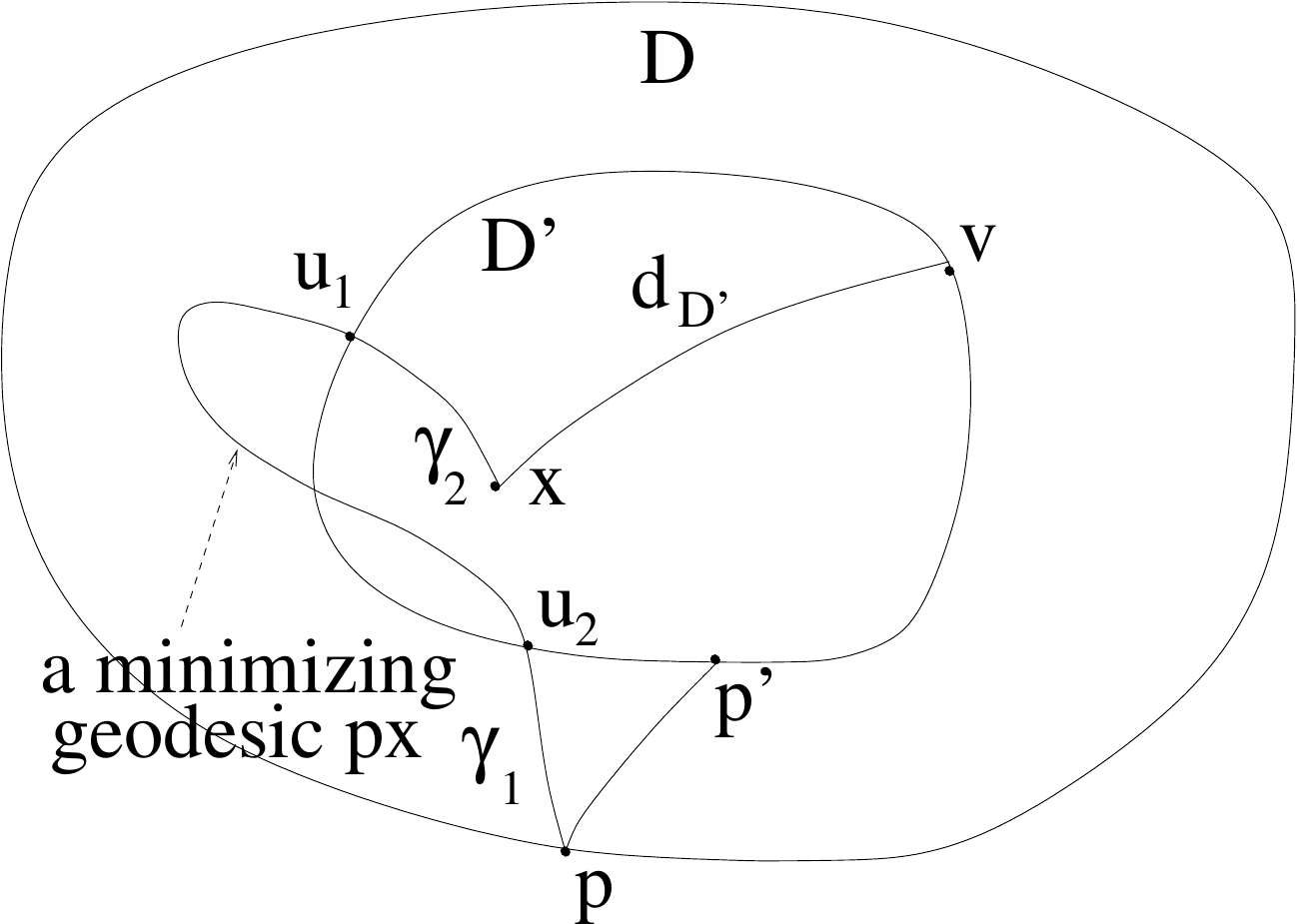}
\caption{}
\label{lnr7}
\end{figure}

Then by triangle inequality
$$|\alpha| \leq |\gamma_1| +  \frac{1}{2} |\partial D'|,$$
$$|\beta| \leq |\gamma_2| + \frac{1}{2} |\partial D'|.$$
Hence,
$d_{D'}+|\alpha| \leq d_D + |\partial D'|.$
\end{proof}

\begin{lemma} \label{diameter2}
Suppose $D' \subset D$ is a subdisc such that $\partial D' \cap \partial D\not =\empty$.
Then $d_{D'} \leq d_D + |\partial D' \setminus \partial D|$.
\end{lemma}
 
\begin{proof}

Note that $\partial D' \setminus \partial D$ is a colection of
countably many open arcs with endpoints on $\partial D$.

Let $\beta$ be a minimizing geodesic in $D'$ from
a point $p \in \partial D'$ to a point $x \in D'$,
such that $|\beta|= d_{D'}$. Let $\alpha$ be a minimizing
geodesic in $D$ from $p$ to $x$ (see Fig. 8).

\begin{figure}[center]
\includegraphics[scale=0.5]{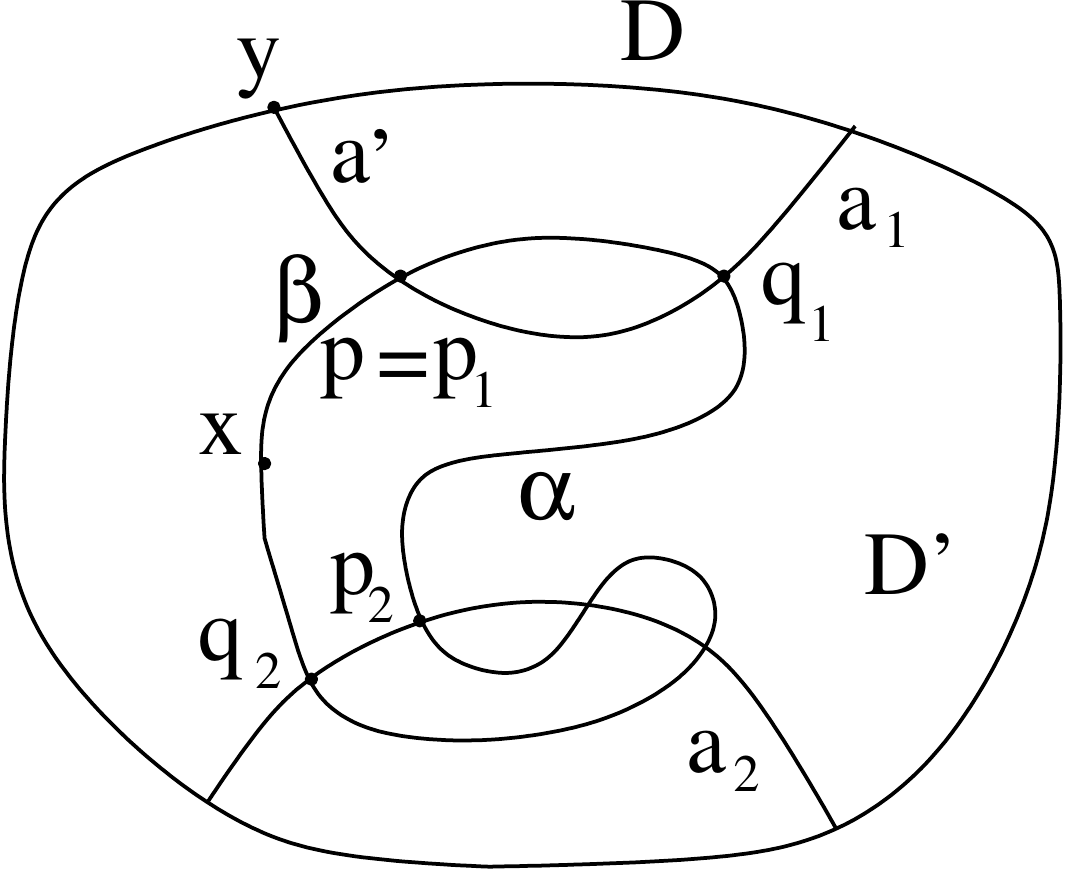}
\caption{}
\label{lnr8}
\end{figure}

We will construct a new curve $\alpha'$ which agrees with $\alpha$
on the interior of $D'$ and lies entirely in the closed disc $D'$.
If $\alpha$ does not intersect any arcs of $\partial D' \setminus \partial D$
we set $\alpha' = \alpha$. Otherwise, let $a_1$ denote the first arc of
$\partial D' \setminus \partial D$ intersected by $\alpha$.
Let $p_1$ (resp. $q_1$) denote the point where $\alpha$ intersects
$a_1$ for the first (resp. last) time. (If $p\in \partial D'\setminus\partial D$, then
$p_1=p$.) We replace the arc of $\alpha$
from $p_1$ to $q_1$ with the subarc of $a_1$. We call this new curve
$\alpha_1$. We find the next (after $a_1$) arc
$a_2 \subset \partial D' \setminus \partial D$ that $\alpha_1$ intersects
and replace a subarc of $\alpha_1$ with a subarc of $a_2$.
We continue this process inductively until we obtain a curve $\alpha'=\alpha_n$ that
lies in $D'$.

Note that $|\beta|\leq |\alpha'| \leq |\alpha|+ |\partial D' \setminus \partial D|$.
Hence, if $p \in \partial D$, then $|\alpha| \leq d_D$ and we are done.

If $p$ belongs to an arc $a \subset \partial D' \setminus \partial D$, then
let $a'$ be a subarc of $a$ connecting $p$ to a point $y$ of $\partial D$,
such that $a' \cap \alpha' = \{ p \}$. (Note, that in this case $p=p_1$.) Then
$\vert\alpha\vert\leq dist_D(x,y)+\vert a'\vert\leq d_D+\vert a'\vert$, and
$d_{D'}=\vert\beta\vert\leq\vert\alpha'\vert\leq\vert\alpha\vert+\vert\partial D'\setminus(\partial D\bigcup a')\vert\leq \vert\alpha\vert+\vert\partial D'\setminus \partial D\vert-\vert a'\vert\leq d_D+\vert\partial D'\setminus \partial D\vert.$
\end{proof}

\section{Proof of Theorem 1.1 A-C.}

We are now ready to prove statements A to C of Theorem \ref{main'}.

Let $\epsilon_0$ be an arbitrary positive number less than $0.001$.
Fix an $\epsilon=\epsilon(\epsilon_0) > 0$ small enough for Lemma \ref{small area}.

Let $N$ be an integer defined by
$$(\frac{4}{3})^{N-1} \epsilon \leq Area(D) < (\frac{4}{3})^N \epsilon $$

Let $\delta < \min\{\epsilon,(\frac{4}{3})^N \epsilon - Area(D)\}$.
For each $n\in \{0,1,...,N \}$
and for every subdisc $D' \subset D$ with $(\frac{4}{3})^{n-1} \epsilon  - \frac{\delta}{2^{N-n+1}}
\leq Area(D') < (\frac{4}{3})^n \epsilon  - \frac{\delta}{2^{N-n}}$ we will show

A. If  $|\partial D'| \leq 2 \sqrt{3} \sqrt{Area(D)}$ then
$$ pdias(D',D) \leq |\partial D'| + \D \sqrt{Area(D')} + 2d_{D'}.$$

B. If $|\partial D'| \leq 6 \sqrt{Area(D)}$ then
$$ pdias(D',D) \leq |\partial D'| + \C \sqrt{Area(D')} + 2d_{D'}.$$

C. If  $|\partial D'| > 6 \sqrt{Area(D)}$ then

$$ pdias(D',D) \leq (1+\epsilon_0) |\partial D'|+
2\lceil \log_{\frac{4}{3}}(\frac{|\partial D'|-4\sqrt{Area(D')}}{2\sqrt{Area(D')}})  \rceil \sqrt{Area(D')}
+ \C \sqrt{Area(D')} + 2d_{D'}$$
$$\leq 2|\partial D'|+\C \sqrt{Area(D')}+2d_{D'}.$$


Passing to the limit as $\epsilon_0\longrightarrow 0$, we will obtain
the assertion of the theorem.

For $n=0$ we have $Area(D') \leq \epsilon$, and so by Lemma \ref{small area} we are done.
Assume the result holds for every integer less than $n$.
By Lemma \ref{boundary} statement C can be reduced to the following statement:

$C'$. For every subdisc $D'' \subset D'$ such that $|\partial D''| \leq 6 \sqrt{Area(D')}$ we have
$$ pdias(D'',D) \leq |\partial D''| + \C \sqrt{Area(D'')} + 2d_{D''}.$$

In particular this implies statement B.
While proving $C'$ we will sometimes make special considerations for the case
$|\partial D''| \leq 2 \sqrt{3} \sqrt{Area(D')}$, which will be needed to prove statement A.

For any $p,q \in \partial D''$ we will construct a homotopy between the two arcs
satisfying this bound.

Let $l_1$ and $l_2$ be two arcs of $\partial D''$ connecting $p$ and $q$.
Let $\overline{D} \subset D''$ by a subdisc satisfying the conclusions of
Proposition \ref{disc subdivision} with $\delta$ equal to our current $\delta$ (defined at the beginning of this section)
divided by $2^{N+2}$.

We have two cases.

\textbf{Case 1.} $\partial \overline{D} \cap \partial D''$ is nonempty.
Then $\partial \overline{D} \setminus \partial D''$ is a collection of arcs $\{ a_i \}$.
For each arc $a_i$ we have a corresponding subdisc $D_i \subset D''\setminus \overline{D}$ with $a_i \subset \partial D_i$
and $Area(D_i) \leq \frac{3}{4} Area(D'') + \frac{\delta}{2^{N-n+2}} < (\frac{4}{3})^{n-1} \epsilon  - \frac{\delta}{2^{N-n+1}}
\leq Area(D')$.

If $l_1 ^i = l_1 \cap \partial D_i$ is
a non-empty arc, we use the inductive assumption to define a path homotopy of
$l_1 ^i$ to $\partial D_i \setminus l_1 ^i$ through curves of length
$$\leq 2 |\partial D_i| + (\frac{\sqrt{3}}{2} \C + 4 \sqrt{3}) \sqrt{Area(D'')} + 2 d_{D''} + O(\delta),$$
$$\leq |\partial D_i| +  \C \sqrt{Area(D'')} + 2 d_{D''} + O(\delta),$$
where we have used Lemma \ref{diameter2} to bound $d_{D_i}$.

This procedure homotopes $l_1$ to a curve $l \subset l_2 \cup \partial\overline{D}$.
Now using the inductive assumption for $\overline{D}$ we continue our
homotopy from $l_1$ to $l_2$ without exceeding
the length bound. (At this stage we will get rid of $\bar D$.) At the end of this
stage it remains only to homotope arcs on $\partial\bar{D}$ to corresponding
arcs of $l_2$ through some of the discs $D_i$. This step is similar to
the already described step involving arcs of $l_1$.

Virtually the same argument proves statement A in this case.

Note that diameter term $d_D$ is not used in an essential way in this case.
Its necessity comes from Case 2.

\textbf{Case 2.} $\partial \overline{D}$ does not intersect $\partial D''$.
Denote $\partial \overline{D}$ by $\gamma$.
$D'' \setminus \gamma$ is the union of an annulus $A$ and an open  disc $\overline{D}$.
Let $\alpha_1$ (resp. $\alpha_2$) be a minimizing geodesic from $p$
(resp. $q$) to $\gamma$. Let $\gamma_i$ denote the arc of
$\gamma$, such that $l_i \cup \alpha_2 \cup -\gamma_i \cup -\alpha_1$
bounds a disc $D_i$ whose interior is in the annulus $A$.
Note that $Area(D_i) \leq \frac{3}{4} Area(D'') + O(\delta)$.

\begin{proposition} \label{square}
A. If $|l_i| \leq 2 \sqrt{Area(D')}+ O(\delta)$ then there is a homotopy from
$l_i$ to $\alpha_1 \cup \gamma_i \cup -\alpha_2$
through curves of length $\leq \D \sqrt{Area(D')} + 2d_{D''}+ O(\delta)$.

B. If $2 \sqrt{Area(D')} < |l_i| \leq 6 \sqrt{Area(D')}+O(\delta)$ then there is a homotopy from
$l_i$ to $\alpha_1 \cup \gamma_i \cup -\alpha_2$
through curves of length $\leq \C \sqrt{Area(D')} + 2d_{D''}+ O(\delta)$.
\end{proposition}

To prove Proposition \ref{square} we will need the following lemma (see Fig. 4).

\begin{lemma} \label{square1}
If $|\partial D_i|> M=\max\{10 \sqrt{3} \sqrt{Area(D')}, 4|l_i|+ 2 \sqrt{3} \sqrt{Area(D')} \}+ O(\delta)$,
then there exists a geodesic $\beta$ of length $\leq \frac{\sqrt{3}}{2} \sqrt{Area(D)} +O(\delta)$
 connecting $\alpha_1$ to $\alpha_2$ such that the endpoints of $\beta$ divide $\partial D_i$
into two arcs of length $\leq \frac{3}{4} |\partial D_i|$.
\end{lemma}

\begin{proof}
We subdivide $\partial D_i$ into four equal subarcs, starting from
point $p$.
By Besicovitch lemma we can connect two opposite arcs by a curve
$\beta$ of length $\leq \frac{\sqrt{3}}{2}\sqrt{Area(D'')} + O(\delta)$.
Now we consider different cases.

Suppose first that $\beta$ connects a point of $\alpha_k$ ($k=1$ or $2$)
with another point of $\alpha_k$. Since $\alpha_k$ is length minimizing
we obtain $\frac{1}{4}|\partial D_i| \leq \frac{\sqrt{3}}{2}\sqrt{Area(D'')} + O(\delta)$
so $|\partial D_i| \leq 2 \sqrt{3} \sqrt{Area(D')} + O(\delta)$.

If $\beta$ connects a point of $l_i$ to another point of $l_i$ then
$|l_i| \geq \frac{1}{4} |\partial D_i|$.
Similiarly, if $\beta$ connects two points of $\gamma_i$ then
$|\partial D_i| \leq 8 \sqrt{3} \sqrt{Area(D'')}+ O(\delta) $.

Suppose $\beta$ connects a point of $l_i$ to a point of $\gamma_i$.
Since $\alpha_1$ and $\alpha_2$ are length minimizing, we must have
$|\alpha_1| + |\alpha_2| \leq |l_i| + 2|\beta|$, so
$|\partial D_i| \leq 2 |l_i| + 2|\beta| + |\gamma_i| \leq 2 |l_i|+
3 \sqrt{3} \sqrt{Area(D')}$.

Suppose $\beta$ connects a point $x$ of $\gamma_i$ and a point $y$ of $\alpha_k$.
Since $\alpha_i$ is a geodesic minimizing distance to the curve $\gamma$,
we conclude that the subarc of $\alpha_i$ between $y$ and $\gamma_i$ 
has length $\leq |\beta|$.
Hence, $\frac{1}{4} |\partial D_i| \leq |\gamma_i|+|\beta|$,
so $|\partial D_i| \leq 10 \sqrt{3} \sqrt{Area(D')} + O(\delta)$.

Now, suppose $\beta$ connects a point of $l_i$ and a point of $\alpha_k$.
Then $\frac{1}{4} |\partial D_i| \leq |l_i|+|\beta|$ yielding
$|\partial D_i|\leq  4|l_i| + 2\sqrt{3} \sqrt{Area(D')}$.

If $|l_i| \leq 2\sqrt{3} \sqrt{Area(D')}$, then in all of the above cases we have
$|\partial D_i| \leq 10 \sqrt{3} \sqrt{Area(D')} + O(\delta)$.
If $|l_i| > 2\sqrt{3} \sqrt{Area(D')}$, then $|\partial D_i|
\leq 4|l_i|+ 2 \sqrt{3} \sqrt{Area(D')}$.

The only remaining case is when $\beta$ connects $\alpha_1$ to $\alpha_2$.
\end{proof}

\textit{Proof of Proposition \ref{square}.}
Proof of B.
Suppose first that $|\partial D_i|\leq M$.

Since $Area(D_i) \leq \frac{3}{4} Area(D'') +\frac{\delta}{2^{N-n+2}}$
and using the inductive assumption we can homotope
$l_i$ to $\alpha_1 \cup \gamma_i \cup -\alpha_2$
through curves of length

$$ \leq (2+\epsilon_0) |\partial D_i|
+ \C \sqrt{Area(D_i)} + 2d_{D_i} + O(\delta).$$

Note that since $|l_i| \leq 6 \sqrt{Area(D')}+ O(\delta)$,
we have $M \leq (24 + 2 \sqrt{3}) \sqrt{Area(D')}+ O(\delta)$
and $M- |l_i| \leq\max\{10\sqrt{3}\sqrt{Area(D')}, 3|l_i|+2\sqrt{3}\sqrt{Area(D')}\}+O(\delta)\leq (18+ 2 \sqrt{3})\sqrt{Area(D')}+ O(\delta)$.
%
%

Therefore, using Lemma \ref{diameter2} the lengths of curves in the homotopy are bounded by
$$\leq |\partial D''|+ (18+2 \sqrt{3} + 24+2 \sqrt{3} + \frac{\sqrt{3}}{2} \C +2(18+ 2 \sqrt{3}))\sqrt{Area(D')} + 2d_{D''} + O(\delta)$$
$$< |\partial D''|+ \C \sqrt{Area(D')} + 2 d_{D''}$$

Note that our choice of the constant $\C>(78+8 \sqrt{3})/(1-{\sqrt{3}\over 2})$ is motivated by the last of these inequalities.

Now consider the case, when $|\partial D_i|> M$. Lemma \ref{square1} implies that
 we can subdivide $D_i$
into two subdiscs $D_i ^1$ and $D_i ^2$ of boundary length $\leq \frac{3}{4} |\partial D_i|
+ \frac{\sqrt{3}}{2}\sqrt{Area(D')} + O(\delta)$ by a curve $\beta_1$ connecting $\alpha_1$ and $\alpha_2$.
For each of subdiscs $D_i ^j$ we have an argument completely analogous to that of Lemma \ref{square1}. We apply it repeatedly until we obtain a sequence of discs $D^k$ stacked on top
of each other with $|\partial D^1| \leq M$ and $|\partial D^{k }| \leq (10\sqrt{3}) \sqrt{Area(D')} + O(\delta)$ for $k\geq 2$ (see Fig. 5).
The discs are separated by Besicovitch geodesics $\{ \beta^k \}$. Let $\alpha_1 ^k$ (correspondingly, $\alpha_2 ^k$)
denote the subarcs of $\alpha_1$ (correspondingly, $\alpha_2$) between $p$ (resp. $q$)
and the endpoint of $\beta^k$.

We homotope $l_i$ to $\alpha_1 ^1 \cup \beta^1 \cup -\alpha_2 ^1$ as described above.
Then we homotope $\alpha_1 ^k \cup \beta^k \cup -\alpha_2 ^k$ to
$\alpha_1 ^{k+1} \cup \beta^{k+1} \cup -\alpha_2 ^{k+1}$ using the inductive assumption
in disc $D^{k+1}$ through curves of length 
$$\leq (2+\epsilon_0)|\partial D^{k+1}|
+\C\sqrt{Area(D^{k+1})}+2 d_{D^{k+1}}+|\alpha_1|+|\alpha_2|$$
$$\leq ((40+10\epsilon_0)\sqrt{3}+{\sqrt{3}\over 2}\C) \sqrt{Area(D')}+ 2 d_{D''} + O(\delta)
<\C \sqrt{Area(D')}+ 2 d_{D''} + O(\delta),$$
where we have used Lemma \ref{diameter} to bound $2 d_{D^{k+1}}+|\alpha_1|+|\alpha_2|$.

The proof of A is analogous with the only difference that both $M$ and
$M-|l_i|$ are majorized by $\leq 10 \sqrt{3}\sqrt{Area(D')}$. The only
purpose of A is to obtain a somewhat better value of the constant at
$\sqrt{Area(D)}$ in Theorems 1.3 A and
Theorem 1.6 A. Therefore we omit the details.

This finishes the proof of Propostion \ref{square}.

Using Propostion \ref{square} we homotope $l_1$ to $\alpha_1 \cup \gamma_1 \cup - \alpha_2$.
Using inductive assumption in the disc $\overline{D}$ and Lemma \ref{diameter} we homotop
$\alpha_1 \cup \gamma_1 \cup - \alpha_2$ to $\alpha \cup \gamma_2 \cup - \alpha_2$. By applying
Proposition \ref{square} again we homotope $\alpha \cup \gamma_2 \cup - \alpha_2$ to $l_2$.
This finishes the proof of statements A to C of Theorem \ref{main'}.
The proof of statement D is presented in the last section.

\section{Subdivision by short curves II.}

In this section we are going to prove the following theorem:
%
%
%

\begin{theorem} \label{diameter subdivision}
A. Let $M$ be a Riemannian $2$-sphere, $p$ a point in $M$. For every
positive $\epsilon$ there exists a simple loop on $M$ of length
$\leq 2 \max_{x\in M}dist(x,p)+\epsilon$ based at $p$ that divides $M$ into two discs
with areas in the interval $(\frac{1}{3}Area(M)-\epsilon^2, \frac{2}{3} Area(M)+\epsilon^2)$.

B. Let $D$ be a Riemannian 2-disc. For every $\epsilon >0$ there exists a
curve $\beta$ of length $\leq 2 \sup_{x\in D}dist(x,\partial D) + \epsilon$
with endpoints on the boundary $\partial D$, which
does not self-intersect and divides $D$ into subdiscs $D_1$ and $D_2$
satisfying
$$\frac{1}{3} Area(D) - \epsilon^2 \leq Area(D_i) \leq \frac{2}{3} Area(D) + \epsilon^2.$$
\end{theorem}

\begin{proof}
\par\noindent
A.
Fix a diffeomorphism $f:S^2\longrightarrow M$.
Consider a very fine triangulation of $S^2$.
We are assuming that the length of the image of each $1$-simplex of
this triangulation under $f$ does not exceed $\epsilon$, and the area of
the image of each $2$-simplex does not exceed $\epsilon^2$.
Extend this triangulation
to a triangulation of $D^3$ constructed as the cone of the chosen triangulation
of $S^2$ with one extra vertex $v$ at the center.
We are going to prove the assertion by contradiction.
Assume that all simple loops of length $\leq 2d+\epsilon$
based at $p$ divide $M$ into two subdiscs one of which has area $\leq {1\over 3}Area(D)-\epsilon^2$.
We are going to construct a continuous extension of $f$ to $D^3$ obtaining
the desired contradiction. We are going to map the center $v$ of $D$ into
$p$. We are going to map each $1$-simplex $[vv_i]$ of the considered
triangulation of $D^3$ to a shortest geodesic connecting $p$ with $f(v_i)$.
We extend $f$ to all $2$-simplices $[vv_iv_j]$ by contracting the triangular loop
formed by the shortest geodesics connecting $p$, $f(v_i)$ and $f(v_j)$ within
one of two discs in $M$ bounded by this loop that has a smaller area.
This disc has area $\leq {1\over 3}Area(D)-\epsilon^2$.
Now it remains to construct the extension of $f$ to the interiors
of all $3$-simplices $[vv_iv_jv_k]$ of the chosen triangulation of $D^3$.
Note that the area of the image of the boundary of such a $3$-simplex  
does not exceed $3({1\over 3}Area(M)-\epsilon^2)+\epsilon^2<Area(M)$.
Therefore the restriction of the already constructed extension of $f$ to
this boundary has degree zero, and, therefore, is contractible. This completes
our extension process and yields the desired contradiction.
\par\noindent
B. We can deduce B from the proof of A by collapsing $\partial D$ into a point
$p$ and repeating the argument used to prove part A for the resulting
(singular) $2$-sphere. Yet one can give another direct proof by contradiction
as follows. Assume that the assertion of the theorem is false. Consider
a very fine geodesic triangulation of the disc. Assume that the areas
of all triangles are less than $\epsilon^2$.
We are going to construct a retraction
$f$ of $D$ onto $\partial D$, thereby obtaining a contradiction as follows:
First we are going to map all new vertices of the triangulation. Each vertex
will be mapped to (one of) the closest points on $\partial D$. Each
edge $v_iv_j$ will be mapped to one of two arcs in $\partial D$ connecting
$f(v_i)$ with $f(v_j)$. We have two possible choices. We choose the arc that
together with the geodesic broken line $f(v_i)v_iv_jf(v_j)$ encloses
a subdisc $D_{ij}$ of $D$ of a smaller area (which is $\leq {1\over 3} Area(D)-\epsilon^2$). Now we need to extend the constructed map to
all triangles $v_iv_jv_k$ of the triangulation. We claim that the three
arcs between $f(v_i), f(v_j)$ and $f(v_k)$ do not cover $\partial D$,
and, therefore, $f(\partial v_iv_jv_k)$ can be contracted within $\partial D$
yielding the desired contradiction. Indeed, in the opposite case the discs $D_{ij}$,
$D_{ik}$ and $D_{jk}$ would cover all $D$ with a possible exception
of a part of the triangle $v_jv_jv_k$ (see Fig. 9).

\begin{figure}[center]
\includegraphics[scale=0.4]{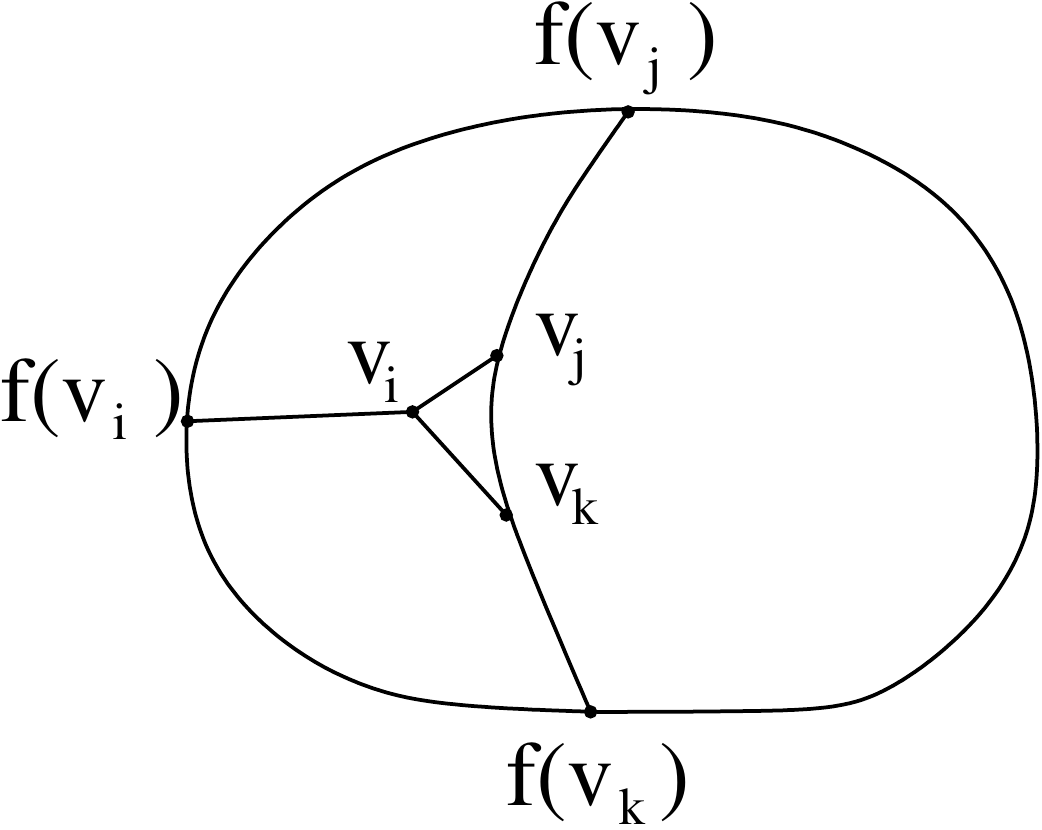}
\caption{}
\label{lnr9}
\end{figure}

But this is impossible as
the sum of their areas does not exceed $Area(D)-3\epsilon^2$ which is strictly
less than $Area(D)-\epsilon^2$.

{\bf Remark.} One can ask for which values of $a\in (0, {1\over 2}]$ 
there exists a constant $c(a)$ with the following property: For each Riemannian
sphere $M$ of diameter $d$ there exists a simple
closed curve of length $\leq  c(a) diam(M)$ dividing
$M$ into two discs of area $\geq a\ Area(M)$.
Theorem 6.1 implies the positive answer for this question
for all $a<{1\over 3}$ with $c(a)=2$. However,
one of the authors
recently proved that the answer for this question is negative if $a={1\over 2}$
([L2]). Moreover, the main result of [L2] asserts that there is no
$C$ such that for every Riemannian $2$-sphere $M$ and $\epsilon>0$ there
exists a finite collection of simple closed curves on $M$
of length $\leq C\ diam(M)$ dividing $M$ into two domains of area
$\geq {1\over 2}\ Area (M)-\epsilon$. On the other hand, Theorem 6.1
easily implies the following result:

\begin{theorem} For each $a\in (0,{1\over 2})$ there exists a constant $C(a)$
such that for every Riemannian $2$-sphere $M$ 
there exists a simple closed curve of length $\leq C(a)diam(M)$
dividing $M$ into two discs of area $\geq a Area(M)$.
\end{theorem}

\begin{proof}
First, use Theorem 6.1 A to divide $M$ into
two discs $D_1, D_2$ of area $\geq Area(M)/3-\epsilon^2$ by a simple closed curve
of length $\leq 2 diam(M)+\epsilon$.  
Apply Theorem 6.1 B to $D_1$ and $D_2$ obtaining discs $D_{ij}, i,j\in\{1,2\}$.
Keep applying Theorem 6.1 B to each of the new discs obtained on the previous
step providing that its area is $\geq ({1\over 2}-a)Area(M)$. We are going
to stop only after the areas of all the
discs will become $<({1\over 2}-a)Area(M)$.
As each application of Theorem 6.1 decreases the area of the disc at least by a factor
of ${2\over 3}-o(1)$, the total number of ``cuts" will be bounded by
a constant $F(a)$ depending only on $a$ (that, obviously, behaves as
$2^{\log_{{2\over 3}-o(1)}({1\over 2}-a)}$, as $\epsilon\longrightarrow 0$). As $\sup_{x\in U}dist(x,\partial U)\leq diam(M)$ for each
subdomain $U\subset M$, the total length of all ``cuts" does not
exceed $F(a)(2diam(M)+\epsilon)$. So, we obtain a cell subdivision
of $M$ into $N\leq F(a)+1$ discs of area $<({1\over 2}-a)Area(M)$. It is
well-known that every cell subdivision of a $2$-sphere into a finite
collection of cells is shellable, that is, we can enumerate these $N$ discs
by integers $1,\ldots, N$ so that for every positive integer $k<N$ the union
of the first $k$ discs and the complement of this union are both homeomorphic to a disc.
It is clear that starting
from the first disc and attaching the discs one by one we will find a value
of $k<N$ such that the union of the first $k$ discs has an area
between $a Area(M)$ and ${1\over 2} Area(M)$.
\end{proof}


\section{Proofs of Theorem 1.2, 1.3 and 1.6 D}

\begin{definition} For each disc $D$ define $\delta_D$ by the formula
$\delta_D = \sup_{x \in D} dist(x, \partial D).$
\end{definition}

Note the following properties of $\delta_D$:

1. $d_D - \frac{|\partial D|}{2}\leq \delta_D \leq d_D$.

2. If $D' \subset D$ then $\delta_{D'} \leq \delta_D$.

%

\begin{lemma} \label{log induction}

For each $n\geq 1$
$pdias (D) \leq 2|\partial D| + 2 d_D + 8 n \delta_D + \C \sqrt{(\frac{2}{3})^n Area(D)}.$

\end{lemma}

\begin{proof}
The proof is by induction on $n$. If $n=0$,
then Theorem \ref{main}  implies that
$pdias(D) \leq 2|\partial D| + 2 d_D + \C \sqrt{Area(D)}$

Suppose the claim is true for $n-1$. Choose $\epsilon>0$ that can later be
made arbitrarily small.
We use Theorem \ref{diameter subdivision} to subdivide
$D$ into two subdiscs of area $\leq \frac{2}{3} Area(D) + \epsilon^2$ by a curve $\beta$ of length
$\leq 2 \delta_D + \epsilon^2$.

The inductive assumption implies that
we can homotope an arc of $l_1$ (from the definition of $pdias$) over each of
$D_i$ via curves of length less than or equal to
$ 2|\partial D|+ 2 |\beta| + 2 d_{D_i} + 8 (n-1) \delta_{D_i} + \C \sqrt{(\frac{2}{3})^{n-1} Area(D_i)}+O(\epsilon).$

We have $d_{D_i} \leq d_D + |\beta|$ by Lemma 4.2. Hence
the lengths of the curves are bounded by
$ 2 |\partial D|+ 8 n \delta_D + 2 d_D + \C \sqrt{(\frac{2}{3})^n Area(D)} + O(\epsilon).$
\end{proof}

The next proposition allows us to somewhat improve our estimates by getting
 rid of an extra $|\partial D|$.
A simple informal idea behind its proof 
is that a search for a path homotopy between two 
complementary arcs in the boundary of $D$ can be reduced to a search for 
a family of path homotopies between pairs of minimizing geodesics (of the
same length) connecting certain pairs of points of $\partial D$. To be more
precise we will be dealing with triangles, where one of sides is very short,
and two other sides are minimizing geodesics; we will be looking for
path homotopies between a pair of sides of the triangle and the remaining side.

\begin{proposition} \label{boundary-diameter}
Suppose that $f(x,y,z)$ is a continuous function such that
$pdias(D) \leq f(|\partial D|, diam(D), Area(D))$
for every disc $D$.  Then
$$pdias(D) \leq \max_{0 \leq t \leq |\partial D|} |\partial D| - t + f(\min \{2(|\partial D| - t),
2t, 2 diam(D) \}, diam(D), Area(D))$$
\end{proposition}

\begin{proof}
Let $p,q$ be the endpoints of $l_1 \cup -l_2 = \partial D$ and $\beta$
be a minimizing geodesic from $p$ to $q$.
We will construct a homotopy from $l_1$ to $\beta$.
We choose a small $\epsilon>0$ and partition $[0,1]$ by $N+1$ points $\{0=a_0,...,a_N=1 \}$
so that $|l_1({[a_i,a_{i+1}]})| \leq \epsilon$. Let $\alpha_i$ denote a minimizing geodesic from
$p$ to $l_1(a_i)$, where $\alpha_{N+1}=\beta$.
Inductively we homotop $\alpha_i \cup l_1([a_i,1])$
to $\alpha_{i+1} \cup l_1([a_{i+1},1])$. More precisely, we are looking
for a path homotopy between $\alpha_i\bigcup l([a_i,a_{i+1}])$ and 
$\alpha_{i+1}$ while keeping $l([a_{i+1}, 1])$ unchanged (see Fig. 10).

\begin{figure}[center]
\includegraphics[scale=0.3]{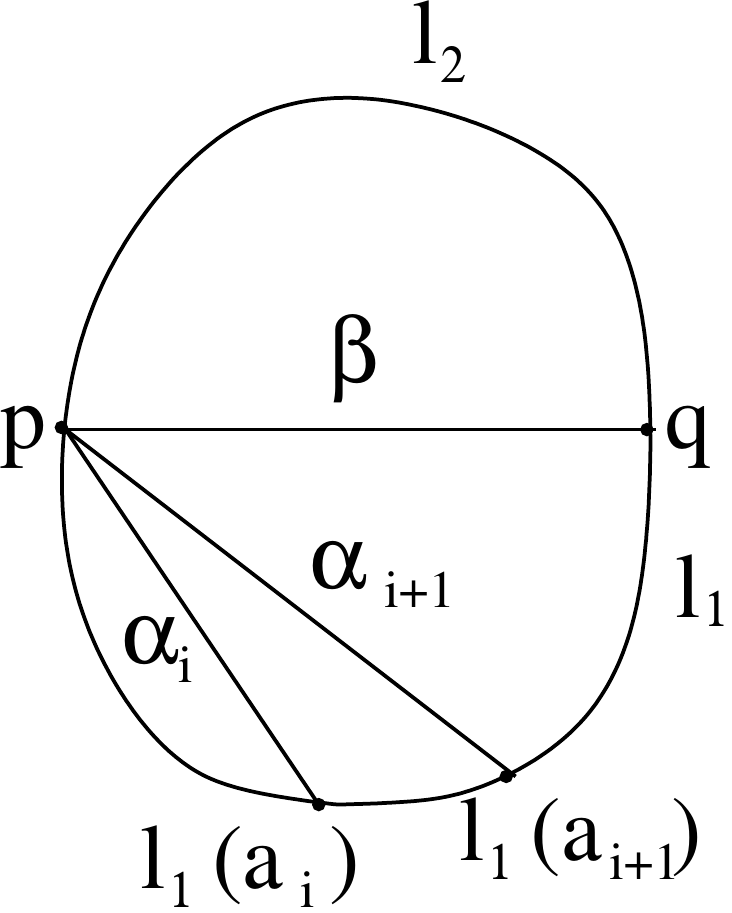}
\caption{}
\label{lnr10}
\end{figure}

Consider the subdisc bounded by $\partial D^i = \alpha_i \cup l_1([a_i,a_{i+1}]) \cup - \alpha_{i+1}$.
Since $\alpha_i$ are length minimizing we have $|\partial D^i| \leq \min \{2(|\partial D| - t)+\epsilon,
2t+\epsilon, 2 diam(D) \} +\epsilon$, where $t=|l_1([0,a_i])|$. Using our assumption we obtain a homotopy
from $\alpha_i \cup l_1([a_i,1])$ to $\alpha_{i+1} \cup l_1([a_{i+1},1])$.
The homotopy between $\beta$ and $l_2$ can be constructed in the same way.
It remains to pass to the limit as $\epsilon\longrightarrow 0$.
\end{proof}

In particular, we can now prove statement D of Theorem \ref{main'}.
From statements B, C we know that

$$ pdias(D) \leq |\partial D|+
2\max\{0,\lceil\log_{\frac{4}{3}}(\frac{|\partial D|-4\sqrt{Area(D)}}{2\sqrt{Area(D)}})  \rceil\} \sqrt{Area(D)}
+ \C \sqrt{Area(D)} + 2d_{D}.$$

Then if we set $L_t = \min \{2(|\partial D| - t),2t, 2 diam(D) \}$ we obtain an estimate

$$pdias(D) < \max_t (|\partial D| - t + \frac{L_t}{2}) + \frac{L_t}{2} +
2\max\{0,\lceil \log_{\frac{4}{3}}(\frac{L_t-4\sqrt{Area(D)}}{2\sqrt{Area(D)}})  \rceil\} \sqrt{Area(D)} $$

$$+ \C \sqrt{Area(D)} + 2 diam(D)$$

$$\leq |\partial D| + 2\max\{0,\lceil \log_{\frac{4}{3}}(\frac{diam(D)}{\sqrt{Area(D)}}-2\rceil)\} \sqrt{Area(D)}
+ \C \sqrt{Area(D)} + 3 diam(D),$$
as $L_t\leq 2diam (D)$ and $-t+{L_t\over 2}\leq 0$.
The formula for excess follows from this estimate.

An analogous coarser estimate that uses Theorem \ref{main} instead of Theorem \ref{main'} C yields

$$pdias(D) \leq |\partial D| + 5 diam (D) + \C \sqrt{Area(D)}.$$

\textit{Proof of Theorem \ref{main2}}

Using Lemma 7.2 
we obtain

$pdias (D) \leq |\partial D| + (5 + 8 n) diam D + \C \sqrt{(\frac{2}{3})^n Area(D)}+O(\epsilon)$

Let $k$ be any positive number, such that  $n=2 \log _{3/2} (\frac{\sqrt{Area(D)}}{ diam(D) k})$
is a natural number. Then the previous estimate can be written as

$$pdias(D) \leq |\partial D| + (\C k + 5 + 16 \log_{3/2} (\frac{1}{k}) +16 \log_{3/2}(\frac{\sqrt{Area(D)}}{diam(D)})) diam(D) $$

Suppose first that $\sqrt{Area(D)} > (\frac{2}{3})^{6.5} diam(D)$.
Note that for some $k \in [(\frac{2}{3})^{7} ,(\frac{2}{3})^{6.5}]$ we will have
$2 \log _{3/2} (\frac{\sqrt{Area(D)}}{diam(D) k}) \in \mathbb{N}$.
It is easy to check that for each $k$ in this interval $\C k + 16 \log_{3/2} (\frac{1}{k}) < 154$.
Hence, from the previous inequality and using $\frac{16}{\ln(3/2)}<40$ we obtain

$$pdias(D) \leq |\partial D| + 159 diam (D) + 40 \ln(\frac{\sqrt{Area(D)}}{diam(D)}) diam(D). $$

If $\sqrt{Area(D)} \leq (\frac{2}{3})^{6.5} diam(D)$, then

$$pdias(D) \leq |\partial D| + 5 diam (D) + \C \sqrt{Area(D)} \leq |\partial D| + 50 diam (D).$$

\textbf{Remark.}
We can obtain a better asymptotic estimate if instead of a bound with $2 |\partial D|$  we use
the one from Theorem \ref{main'} with the logarithmic term. Then for $\frac{\sqrt{Area(D)}}{diam(D)} \rightarrow \infty$
we obtain $pdias(D) < |\partial D| + ({12\over\ln{3\over 2}}+o(1))\ln(\frac{\sqrt{Area(D)}}{diam(D)}) diam(D)$. (Note that ${12\over\ln{3\over 2}}=29.5956\ldots$).

Note that the 25 percent improvement of the constant at $diam(D)\ln{\sqrt{Area(D)}\over diam(D)}$ (from ${16\over\ln{3\over 2}}$ to ${12\over\ln{3\over2}}$) comes from the fact
that the term $2|\beta|$ in the proof of Lemma 7.2 can be replaced by $|\beta|$,and $8n\delta_D$ in the right hand side in the inequality of Lemma 7.2 becomes
$6n\delta_D$.

\textit{Proof of Theorem \ref{sphere}.}

Let $p$ be an arbitrary point of $M$. Take the metric ball $B_\epsilon(p)$
of a very small radius $\epsilon$ centered at $p$ and choose a point 
$q\in\partial B_\epsilon(p)$. Applying Theorem 1.6 A we see that
one can contract $\partial B_\epsilon(p)$ in $M\setminus B_\epsilon(p)$
as a loop based at $q$ via loops of length not exceeding the right hand side
in Theorem 1.3 A plus $O(\epsilon)$. Now we can attach two copies of the
geodesic segment $(pq)$
connecting $p$ and $q$ at the beginning and the end of each of those loops
based at $q$. As the result, we will obtain a family of loops based
at $p$. Finally, add a family of loops based at $p$ that constitutes a homotopy
between the constant loop $p$ and $(pq)*\partial B_\epsilon(p)*(qp)$
and a family of loops that contracts $(pq)*(qp)$ over itself to the constant
loop $p$. The lengths of all these new loops are $O(\epsilon)$. As the result,
we obtain a family of loops based at $p$ of
lengths $\leq 664\sqrt{Area(M)}+2 diam(M) +O(\epsilon)$
that sweeps-out $M$. Now pass to the limit as $\epsilon\longrightarrow 0$.

To prove inequality B we can proceed as above with the only difference that
$\partial B_\epsilon(p)$ will be contracted in $M\setminus B_\epsilon(p)$
using Theorem 1.2 instead of Theorem 1.6 A.
%
%
Finally, note that, when ${\sqrt{Area(M)}\over diam(M)}\longrightarrow\infty$,
one can improve the constant in inequality B exactly as it had been
described in the remark after the proof of Theorem 1.2 above.
The result will be the last assertion in Theorem 1.2.
\end{proof}

\section{Near-optimality of upper bounds for $exc(d,A)$ and $dias(M)$. (Proof of Theorems AA and BB.)}

In this section we will use a modification of the examples from \cite{FK} to prove Theorems AA and BB.
The idea of these examples is to embed a large tree in a disc $D$ 
and to define a Riemannian metric on $D$ so that the distance between 
non-adjacent edges of the tree is comparable to the diameter of the disc,
and the area of the disc is proportional to the number of edges. One observes then
that every homotopy $(\gamma_t)_{t\in [0,1]}$ contracting the boundary $\gamma_0$ of the disc to a point
will contain a closed curve $\gamma_{t_*}$ that intersects at least $n$ edges of the tree, where $n$ is proportional
to the height of the tree.

To obtain a better bound we will somewhat modify the construction in \cite{FK}.
We will use ternary trees instead of binary trees and define a metric on $D$ which is flat 
away from the tree. Similar construction appeared in a work of Riley and Thurston
\cite{RT}.

Define a ternary tree of height $n$ inductively as follows.
Ternary tree of height $0$ is a point. A ternary tree of height $n$ is obtained by 
taking a ternary tree of height $n-1$ and attaching three new edges to each vertex that 
has less than three adjacent edges. 
A ternary tree of height $n$ has $\frac{3}{2}(3^n - 1)$ edges. We regard a ternary tree as a metric $1$-complex
endowed with a metric, such that the length of each edge of the tree is equal to $1$.

Let $N=\frac{3}{2}(3^n - 1)$ and consider a $2N-$gon $Q \subset \mathbb{R}^2$. Connect each vertex of $Q$ to
a point $a$ inside $Q$ by a line segment. On each of $2N$ triangles obtained in this way
we define a metric that turns it into a flat equilateral triangle. 
We obtain a flat metric on $Q$ with a singularity at $a$ (the total angle around 
$a$ is $\frac{2 N \pi }{3}$). We will now glue pairs of edges of $Q$
so that they will form a ternary tree $T$.
More precisely, consider a closed tour of the tree $T$ so that each edge
is traversed twice, and regard this tour as a quotient map $q:S^1 = \partial Q \rightarrow T$.
Consider the quotient space obtained from the disjoint union of $Q$ and $T$ by identifying each point $z\in\partial Q$ with $q(z)$.
The resulting quotient space will be a Riemannian $2$-sphere which is singular at a point $a$
and along $T$. We remove a very small metric disc $D_0$ around $a$ with diameter, area and the length of boundary 
$< \epsilon$ for some very small positive $\epsilon$. As a result we obtain a Riemannian $2$-disc $D$ with a boundary of length $<\epsilon$. 

Let $c$ denote the length of the side of the equilateral triangle
used in the construction. 
If two points $x,y \subset T \subset D$ are at a distance $\geq 1$ in $T$ 
(with intrinsic tree metric)
then the distance between them in $D$ is at least ${\sqrt{3}\over 2}c$.
Also, as $N \rightarrow \infty$ and $\epsilon \rightarrow 0$
the diameter $d$ of $D$ approaches $2c$ and the area is asymptotic to 
$\frac{\sqrt{3}}{8}N d^2$.
%
%
%

We claim that for every homotopy $(\gamma_t)_{t\in [0,1]}$ contracting the boundary of $D$
to a point there will be a closed curve $\gamma_{t_*}$ of length $\geq {\sqrt{3}\over 2} n c $.
As a corollary, it would immediately follow that $exc$ function satisfies

$$exc(d,A)\geq \frac{1}{ \ln(3)} d \ln (\frac{\sqrt{A}}{d}) (1- o(1)),$$

when ${\sqrt{A}\over d}\longrightarrow\infty$.

To prove this we claim will use Theorem 0.6 from \cite{CR} asserting that
if there exists a homotopy $\gamma_t$
from a boundary of a Riemannian $2$-disc $D$ to a point that passes through curves of length $\leq L$, 
then for each $\epsilon>0$ there exists a diffeomorphism 
$g:D \rightarrow D_{st}$ from $D$ to the standard Euclidean disc, 
such that the pre-image of each concentric circle has length $\leq L + \epsilon$.
In other words, we can always replace an arbitrary contraction by a monotone contraction
increasing the length of the longest curve by at most $\epsilon$ for an arbitrarily small positive $\epsilon$.
The following lemma then implies that at least one of the curves in a monotone contraction 
must be long.

\begin{lemma} \label{ternary_tree}
Let $T$ be a ternary tree of height $n$ and $f:T \rightarrow \mathbb{R}$ a continuous
function. Then for some point $x \in \mathbb{R}$ the level set $f^{-1}(x)$
contains $n$ points $x_1,\ldots ,x_n$ such that the distance between each pair $x_i$, $x_j$ of these
points in the tree metric is at least $1$.
\end{lemma}

\begin{proof}
The proof is by induction on the height of $T$.
The claim is trivially true for a ternary tree of height $1$.
Assume it to be true for a tree of height $n-1$.

Let $A = \max \{ f(T) \}$ and $B = \min \{ f(T) \}$.
Consider a geodesic path $p \subset T$
between a point $a' \in f^{-1} (A)$ and a point $b'  \in f^{-1} (B)$.
We claim that there exists a ternary subtree $T' \subset T$ of height
$n-1$, such that $dist(T',p) \geq 1$. 
Indeed, let $T_i$ denote three ternary subtrees
that can be obtained from $T$ by removing the root and three
edges adjacent to it. Now we observe that every geodesic path in $T$ intersects
at most two of the three trees $T_i$.

By the inductive assumption there exists a point $u\in [A,B]$  such that $f^{-1}(u) \cap T'$
contains $n-1$ points $x_1,\dots , x_{n-1}$ such that all pairwise distances between these points are not less than $1$.
Definition of $p$  immediately implies that $f^{-1}(u) \cap p$ is non-empty. Now we can choose any point of
$f^{-1}(u)\cap p$ as $x_n$.
\end{proof}

Lemma \ref{ternary_tree} holds for binary trees, but with $n$ replaced by $n/2$
(see \cite{FK}). This leads to a somewhat worse bound than if we were using a ternary tree.

Thus, we have established the first part of Theorem AA. In order to prove the second part of Theorem AA we are going
to modify the above construction of the metric on the $2$-disc as follows.  First, choose the length of the boundary of the $2$-disc
$\epsilon$ so that $\delta=d\epsilon$ is much smaller than $A$ (i.e. $\delta =A o(1)$, as ${\sqrt{A}\over d}\longrightarrow 0$.)
Now take $n=5$ (and, therefore, $N=363$), and $c=\sqrt{{2(A-\delta)\over 363\sqrt{3}}}.$ Now attach to the boundary of the disc the cylinder
of length $d-c$. In other words, we scale previously obtained metrics to a small (but not negligibly small) size to ensure that
the area is slightly less than $A$ and attach a very long but infinitesimally narrow ``straw" of a negligibly small area to
ensure that the diameter is arbitrarily close to $d$. Now note that the proof of the quoted above Theorem 0.6 from [CR]
implies that if there exists a homotopy contracting the boundary of the disc to a point  via {\it based} loops of length $<L$,
then the boundary can be contracted to a point via a family of simple loops of length $<L$ based at the same point that pairwise intersect
only at the base point. Now Lemma 8.1 implies that if we would try to contract the boundary of the constructed $2$-disc
to a point as a loop based at one of points of its boundary, then the homotopy obtained as an application of this version
of Theorem 0.6 of [CR] will contain a simple curve intersecting the tree at $n=5 $ points. Therefore, the length of this loop will
be at least $(2d-2c)+4(\sqrt{3}/2)c=d+(2\sqrt{3}-2)c=2d+0.08257\ldots\sqrt{A}$.

In order to prove the first part of Theorem BB we ``cap" the metrics on $2$-discs used to prove the first part of Theorem AA by flat round $2$-discs of radius
${\epsilon\over 2\pi}$ attached to the boundary. Now we use the main result of [CL] that asserts that given a sweep-out of a $2$-sphere by
closed curves of length $<L$ for some $L$ there is a slicing of the $2$-sphere by pairwise non-intersecting simple curves of length $<L$.
More specifically, one can find a smooth map of the $2$-sphere into $[0,1]$ such that inverse images of $0$ and $1$ are points, and
the inverse image of each $t\in (0,1)$ is a simple closed curve of length $<L$.
We use this result instead of the result from [CR] used in the proof of Theorem AA. (Note, however, the quoted theorem from [CR] is a part
of the proof of this theorem from [CL].) In the situation of the first part of Theorem BB we 
can conclude that one of these simple closed curves must intersect the tree at $n$ edges, and therefore $dias\geq L\geq {\sqrt{3}\over 2}cn$,
which immediately implies the theorem.
\par
In order to prove the second part of Theorem BB we start similarly to the proof
of the second part of Theorem AA. Namely, we take the metric on the $2$-disc
constructed at the beginning of the first part of Theorem AA, choose
$n=5$, and then choose $c$ so that the area of the resulting metric disc would be somewhat less than $Area(M)/3$. (This is somewhat different from the proof
of the second part of Theorem AA, where we wanted the result to have area somewhat less
than $A$.)  The length of the boundary is negligibly small. Then we attach to the boundary
the cylinder of length $diam(M)/2-c$,
and obtain a metric on the $2$-disc of
diameter very close to $d/2$ and area slightly less than $Area(M)/3$.
Then we take three copies of this metric $2$-disc and ``wield" their thin
ends together. 
Near the point where these three thin ends merge
the resulting Riemannian
$2$-sphere will look like a tubular neighborhood of three rays emanating from
the same origin in a plane and forming angles $2\pi/3$ with each other. The
area of the resulting Riemannian $2$-sphere will be equal to $Area(M)$, and
its diameter can be made arbitrarily close to $diam(M)$. Most of the area
of this $2$-sphere will be concentrated in three congruent ``thick"
metric discs glued at the ends of thin cylindrical ``tubes" that are far
from the thin central area. Each of these three ``thick" discs would
contain a copy of the metric tree of height $n=5$.
\par
Now we going to use again the main result of [CL]. For an arbitrarily small $\epsilon$ one can map the $2$-sphere into $[0,1]$ so that the pre-images of 
$0$ and $1$ are points, and for each $t\in (0,1)$ its inverse image is
a simple curve of length $\leq dias(M)+\epsilon$. Note that for each of the
three trees of height $n$ there exists $t_i\in (0,1)$, $i=1,2,3$, such that its inverse
image $\gamma_i$ intersects the tree at $5$ points such that the pairwise
distances between them are greater than or equal to ${\sqrt{3}\over 2}c$.
Therefore the length of an arc of $\gamma_i$ that starts and ends
at two of these points and contains the other three is at least $2\sqrt{3}c$.
Assume that $\gamma_i$ would pass through the central area $C$ that we imagine
as being arbitrarily small yet such that its complement splits into three
connected components. In this case the length of $\gamma$ would be at least
$d+(2\sqrt{3}-2)c-\delta$, where $\delta$ could be made arbitrarily small.
Substituting $c={\sqrt{2}\over \sqrt{1089\sqrt{3}}}\sqrt{Area(M)}$
we would obtain the
the desired inequality.

So, it remains to prove that at least one of the curves $\gamma_i$ passes
through the central area. Recall that they all are pairwise isotopic through
simple pairwise non-intersecting curves that cannot enter the ``outer" discs
bounded by $\gamma_i$ (not containing the central area). But 
curves $\gamma_i$ will not be even pairwise homotopic in the complement
of the union of these three discs. This completes the proof of the second
part of Theorem BB.

\section{Curvature-free upper bounds for lengths of geodesics on Riemannian
$2$-spheres.}

Let $M$ be a Riemannian $2$-sphere of diameter $d$. It had been proven in
[NR1] that for each $x\in M$ and each $k\geq 1$ there exist $k$ 
distinct non-trivial geodesic loops on $M$ based at $x$ of length $\leq 20kd$. Moreover, one of the authors recently
proved that for each pair of points $x, y\in M$ and each $k\geq 1$ there
exist $k$  distinct geodesics connecting $x$ and $y$ of length $\leq 22(k-1)d+dist(x,y)$ ([NR3]).
It is well-known (and had been noticed in [NR1]) that the classical
A. Schwartz's proof ([Sch]) of J.P. Serre's theorem asserting the existence of an infinite
set of distinct geodesics connecting an arbitrary pair of points $x,y$
of a closed
Riemannian manifold implies that the lengths of $k$ of these geodesics
can be bounded by $2(k-1)Bdias_p(M)+dist(x,y)$ (in the notations of the present paper). The length of a shortest non-trivial geodesic loop based at $p$
does not exceed $Bdias_p(M)$. Combining these known
facts with Theorem 1.3 A we obtain the following result:

\begin{theorem}
Let $M$ be a Riemannian $2$-sphere of diameter
$d$ and area $A$, $x,y$ a pair
of points on $M$ and $k$ a positive integer. Then there exist at least
$k$ distinct geodesics of length $\leq 4(k-1)d + 1328(k-1)\sqrt{A}+dist(x,y)$ connecting $x$ and $y$. In particular, there exist $k$ distinct
non-trivial geodesic loops of length $\leq 4kd+1328k\sqrt{A}$ based
at $x$. The length of the shortest geodesic loop based at any prescribed
point of $M$ does not exceed $2d+664\sqrt{A}$.
\end{theorem}

When ${\sqrt{A}\over d}$ is very small, this theorem provides
a substantial improvement of the quoted upper bounds from [NR1] and [NR3].
Note, that the length of the shortest geodesic loop based at a pole of a very long and thin ellipsoid of revolution is equal to $2d$. Moreover, it seems
plausible that the approach of [BCK] can be used to construct
Riemannian manifolds arbitrarily close to thin ellipsoids of revolution,
such that the length of the shortest geodesic loop based at a pole is strictly greater than $2d$. Therefore, our upper bound for the length of a shortest
geodesic loop seems to be quite good (when
${\sqrt{A}\over d}\longrightarrow 0$).
\par
Similarly, it had been proven in [NR2] and [LNR2] that there exist three distinct non-trivial simple periodic geodesics on $M$ of length $\leq 5d$, $\leq 10d$ and
$\leq 20d$, correspondingly.
On the other hand our proof of Theorem 1.3 yields a {\it homeomorphism}
between the round sphere $S^2\subset {\bf R}^3$ and $M$
sending the South pole of $S^2$ to any
prescribed point $p\in M$ and all circles in $S^2$ passing through the
South pole and obtained as a section of $S^2$ by a plane parallel to $Y$-axis
to simple closed curves of length $\leq 664\sqrt{Area(M)}+2diam(M)$
(that intersect only at $p$).
In [LNR2] we used this estimate and the proof of the classical Lyusternik-Schnirelman theorem asserting the existence of three simple non-trivial periodic geodesics on each Riemannian $2$-sphere
to demonstrate that the lengths of these three geodesics do not exceed, correspondingly,
$700\sqrt{Area(M)}+diam(M)$,
$1400\sqrt{Area(M)}+2diam(M)$, and $2800\sqrt{Area(M)}+4diam(M)$
(see Theorem 1.3 in [LNR2]). It is interesting that the first two of these estimates are seemingly optimal
up to a constant factor at $\sqrt{Area(M)}$,
when ${\sqrt{Area(M)}\over d}\longrightarrow 0$
(see the discussion after the text of Theorem 1.3 in [LNR2]).
\par
Our paper [LNR2] contains another application of the results of the present paper to simple periodic geodesics on Riemannian $2$-spheres which is relevant
for all Riemannian $2$-spheres and not only for ``thin" ones.
The classical proofs
of the Lyusternik-Shnirelman theorem use the non-trivial $1$- , $2$- and $3$-dimensional homology classes of the space of simple non-parametrized curves with coefficients
in $\mathbb{Z}_2$. The results of our present paper can be used to estimate the widths of these three homology classes and, therefore, the lengths of
three simple periodic geodesics that appear in standard proofs of the Lyusternik-Shnirelman theorem. Indeed, Theorem 1.2 of [LNR2] asserts that the lengths of these
three simple periodic geodesics (that all have non-zero indices) can be majorized by $800diam(M)\max\{1,\ln {\sqrt{Area(M)}\over diam(M)}\}$. However,
the quoted above upper bounds $5d$, $10d$ and $20d$ from [LNR2] for the lengths of some three simple periodic geodesics on $M$ are not necessarily  upper bounds for the lengths of these
three simple periodic geodesics with positive indices but can be estimates
for lengths of some simple periodic geodesics of index zero. In fact, we
beleive that one can use examples of Riemannian metrics on $S^2$ similar to those used in [FK], [L], [L2] and the previous section to demonstrate that
the smallest length of a simple periodic geodesics of a positive index cannot be majorized in terms of $diam(M)$ alone.

\medskip\noindent
{\bf Acknowledments.} This work was partially supported by NSERC Discovery
Grants of two of its authors (A.N. and R.R.).

\end{document}